\documentclass[a4paper]{amsart}
\usepackage{amsfonts}
\usepackage{amsmath,amsthm,amssymb}
\usepackage[abs]{overpic}
\usepackage{graphicx}
\usepackage{color}

\newtheorem{theorem}{Theorem}[section]
\newtheorem{lemma}[theorem]{Lemma}
\newtheorem{proposition}[theorem]{Proposition}
\newtheorem{corollary}[theorem]{Corollary}
\theoremstyle{definition}
\newtheorem{definition}[theorem]{Definition}
\theoremstyle{remark}
\newtheorem{remark}[theorem]{Remark}
\newtheorem{example}[theorem]{Example}

\newtheorem*{acknowledgements}{Acknowledgements} 
\numberwithin{equation}{section}
\newtheorem*{prop}{Proposition}
\newtheorem*{thm}{Theorem}
\newtheorem*{cor}{Corollary}

\newcommand{\omu}{\overline{\mu}}
\makeatletter

\@addtoreset{figure}{section}
\makeatother
\begin{document} 

\title{Milnor invariants and edge-homotopy classification of clover links}
\author[Kodai Wada]{Kodai Wada} 
\address{Department of Mathematics, School of Education, Waseda University, Nishi-Waseda 1-6-1, Shinjuku-ku, Tokyo, 169-8050, Japan}
\email{k.wada@akane.waseda.jp}
\keywords{$C_{k}$-equivalence; Clover link; Edge-homotopy; Milnor invariant.}
\date{\today}
\maketitle
\begin{abstract} 
Given a clover link, we construct a bottom tangle by using a disk/band surface of the clover link.
Since the Milnor number is already defined for a bottom tangle, we define the Milnor number 
for the clover link to be the Milnor number for the bottom tangle and show that for a clover link, if Milnor numbers of length $\leq k$ vanish, then Milnor numbers of length $\leq 2k+1$ are well-defined.
Moreover we prove that two clover links whose Milnor numbers of length $\leq k$ vanish are equivalent 
up to edge-homotopy and $C_{2k+1}$-equivalence if and only if those Milnor numbers of length 
$\leq 2k+1$ are equal.
In particular, we give an edge-homotopy classification of $3$-clover links by their Milnor numbers of length $\leq 3$.
\end{abstract} 

\section{Introduction}
In 1954, J. Milnor \cite{M} introduced a concept of {\it link-homotopy} which is a weaker equivalence relation than link type, where link-homotopy is an equivalence relation generated by crossing changes on the same component.
And he \cite{M}, \cite{M2} defined {\it Milnor $\omu$-invariants} which are given as follows.
(See Subsection \ref{Milnor invariants} for detail definitions.)
Let $L$ be an oriented ordered $n$-component link in $S^{3}$.
The {\it Milnor number $\mu_{L}(I)$} is an integer determined by a finite sequence $I$ of numbers in $\{ 1,2,\ldots,n\}$.
Let $\Delta_{L}(I)$ be the greatest common devisor of $\mu_{L}(J)$'s, where $J$ is obtained from proper subsequence of $I$ by permuting cyclicly.
The Milnor $\omu$-invariant $\omu_{L}(I)$ is the residue class of $\mu_{L}(I)$ modulo $\Delta_{L}(I)$.
The length of $I$ is called the {\it length} of $\omu_{L}(I)$ and denoted by $|I|$.

Milnor \cite{M} gave a link-homotopy classification for 2- or 3-component links by Milnor $\omu$-invariants.
In 1988, J. P. Levine \cite{L2} gave a link-homotopy classification for 4-component links.
In 1990, N. Habegger and X. S. Lin \cite{HL} gave an algorithm which determines if two links with arbitrarily many components are link-homotopic.
In \cite{HL}, they defined Milnor numbers for string links which are similarly defined as links and proved that Milnor numbers are invariants for string links.
Moreover they showed that Milnor numbers give a link-homotopy classification for string links with arbitrarily many components.
We remark that Milnor numbers are complete link-homotopy invariants for string links, but Milnor $\omu$-invariants are not strong enough to classify for links up to link-homotopy.

An embedded graph in the $3$-sphere $S^{3}$ is called a {\it spatial graph}.
Let $C_{n}$ be a graph consisting of $n$ loops, each loop connected to a vertex by an edge.
We call a spatial graph of $C_{n}$ an {\it $n$-clover link} in $S^{3}$ \cite{L}.
Given an $n$-clover link $c$, we construct an $n$-component bottom tangle $\gamma_{F_{c}}$ by using a {\em disk/band surface} $F_{c}$ of $c$.
In \cite{L}, Levine defined the Milnor number for a bottom tangle.
Therefore we define the {\it Milnor number $\mu_{c}$ for an $n$-clover link $c$} to be the Milnor number $\mu_{\gamma_{F_{c}}}$. 
In \cite{L}, a bottom tangle is called a string link.
(The name \lq bottom tangle\rq ~follows K. Habiro \cite{H2}.)
In this paper, we mean that a string link is one defined in \cite{HL}.
There is a one-to-one correspondence between the sets of string links and bottom tangles. (See subsection~\ref{tangles}.)
This correspondence naturally induces the one-to-one correspondence between
the Milnor number for the bottom tangles and the Milnor number for the string links.

We remark that there are infinitely many choices of $\gamma_{F_c}$ for $c$, 
and hence that, in general, $\mu_c$ is not an invariant for $c$. But under a certain condition, 
$\mu_c$ is well-defined as follows.

\begin{thm}[Theorem \ref{well-definedness}]
Let $c$ be an $n$-clover link and $l_{c}$ a link which is the disjoint union of loops of $c$.
If $\omu_{l_{c}}(J)=0$ for any sequence $J$ with $|J|\leq k$, then $\mu_{c}(I)$ is well-defined for any sequence $I$ with $|I|\leq 2k+1$.
\end{thm}

In 1988, Levine \cite{L} already defined Milnor numbers for {\it flat vertex} clover links and proved the same result as the theorem above, where a flat vertex spatial graph is a spatial graph $\Gamma$, for each vertex $v$ of $\Gamma$, there exist a neighborhood $B_{v}$ of $v$ and a small flat plane $P_{v}$ such that $\Gamma\cap B_{v}\subset P_{v}$ \cite{Yamada}.
In this paper, we do not assume that clover links are flat vertex ones. 
And we consider that two clover links are {\em equivalent} if 
they are ambient isotopic.

By using Milnor numbers for clover links, we have the following results for an edge-homotopy classification of clover links, where edge-homotopy is an equivalence relation generated by crossing changes on the same spatial edge.
This equivalence relation was introduced by K. Taniyama \cite{T} as a generalization of link-homotopy.

\begin{thm}[Theorem \ref{mainthm}]
Let $c, c'$ be two $n$-clover links and $l_{c}, l_{c'}$ links which are
disjoint unions of loops of $c, c'$ respectively.
Suppose that $\omu_{l_{c}}(J)= \omu_{l_{c'}}(J)=0$ for any sequence $J$ with $|J| \leq k$. 
Then $c$ and $c'$ are (edge-homotopy$+C_{2k+1}$)-equivalence if and only if $\mu_c(I)= \mu_{c'}(I)$ for any non-repeated sequence $I$ with $|I| \leq 2k+1$, where (edge-homotopy$+C_{2k+1}$)-equivalence is an equivalence relation obtained by combining edge-homotopy and $C_{2k+1}$-equivalence which is defined by Habiro \cite{H}.
\end{thm}

We also have the following proposition.
\begin{prop}[Corollary \ref{cor}]
Let $c, c'$ be two $n$-clover links and $l_{c}, l_{c'}$ links which are
disjoint unions of loops of $c, c'$ respectively.
Suppose that $\omu_{l_{c}}(J)= \omu_{l_{c'}}(J)=0$ for any sequence $J$ with $|J| \leq {n}/{2}$. Then $c$ and $c'$ are edge-homotopic if and only if $\mu_c(I)= \mu_{c'}(I)$ for any non-repeated sequence $I$ with $|I| \leq n$.
\end{prop}

It is the definition that the Milnor $\omu$-invariant of length $1$ is zero.
If $n=3$, then the proposition above holds without the condition.
\begin{cor}[Corollary \ref{corn=3}]
Two 3-clover links $c$ and $c'$ are edge-homotopic if and only if $\mu_{c}(I)=\mu_{c'}(I)$ for any non-repeated sequence $I$ with $|I|\leq 3$.
\end{cor}

\begin{remark}
Let $c$ be a clover link and $l_c$ a link which is the disjoint union of loops of $c$. 
By the definition of $\mu_c$ in subsection \ref{Milnor invariants}, we note that 
$\mu_{c}(I)=0$ for any non-repeated sequence $I$ if and only if 
$\mu_{l_c}(I)=0$ for $I$.
And we also note that all Milnor numbers for a trivial clover link vanish, where a clover link is {\em trivial} if there is an embedded plane in $S^3$ which contains the clover link.
It is shown by Milnor~\cite{M} that a link is link-homotopic to a trivial link if and only if the Milnor $\omu$-invariant vanishes for any non-repeated sequence.
Hence, by the proposition above, we have that $c$ is edge-homotopic to a trivial clover link if and only if $l_{c}$ is link-homotopic to a trivial link.
\end{remark}


\begin{acknowledgements}
I would like to express my best gratitude to Professor Akira Yasuhara for his helpful advice and continuous encouragement.
I also would like to thank Professor Kokoro Tanaka for his useful comments.
\end{acknowledgements}

\section{Milnor numbers for clover links}
In this section we will define Milnor numbers for clover links.
\subsection{Tangles}
\label{tangles}
An $n$-component {\it tangle} is a properly embedded disjoint union of $n$ arcs in the $3$-cube $[0,1]^{3}$.
An $n$-component tangle $sl=sl_{1}\cup sl_{2}\cup \cdots \cup sl_{n}$ is an $n$-component {\it string link} if for each $i\ (=1,2,\ldots,n)$, the boundary $\partial sl_{i}=\{ (\frac{2i-1}{2n+1},\frac{1}{2},0), (\frac{2i-1}{2n+1},\frac{1}{2},1)\} \subset \partial [0,1]^{3}$.
In particular, $sl$ is {\it trivial} if for each $i\ (=1,2,\ldots,n)$, $sl_{i}=\{(\frac{2i-1}{2n+1},\frac{1}{2})\}\times[0,1]$ in $[0,1]^{3}$.

Product of $n$-component string links is defined as follows.
Let $sl=sl_{1}\cup sl_{2}\cup\cdots \cup sl_{n}$ and $sl'=sl'_{1}\cup sl'_{2}\cup\cdots \cup sl'_{n}$ be two string links in $[0,1]^{3}$.
Then the {\it product $sl*sl'=(sl_{1}*sl_{1}')\cup (sl_{2}*sl'_{2})\cup \cdots \cup (sl_{n}*sl'_{n})$} of $sl$ and $sl'$ is a string link in $[0,1]^{3}$ defined by 
\begin{center}
$sl_{i}*sl'_{i}=h_{0}(sl_{i})\cup h_{1}(sl'_{i})$
\end{center}
for $i=1,2,\ldots,n$, where $h_{0},h_{1}:([0,1]\times[0,1])\times[0,1]\rightarrow ([0,1]\times[0,1])\times[0,1]$ are embeddings defined by 
\begin{center}
$h_{0}(x,t)=(x,\frac{1}{2}t)$ and $h_{1}(x,t)=(x,\frac{1}{2}+\frac{1}{2}t)$
\end{center}
for $x\in([0,1]\times[0,1])$ and $t\in[0,1]$, see Figure~\ref{product}.

\begin{figure}[htpb]
 \begin{center}
  \begin{overpic}[width=10cm]{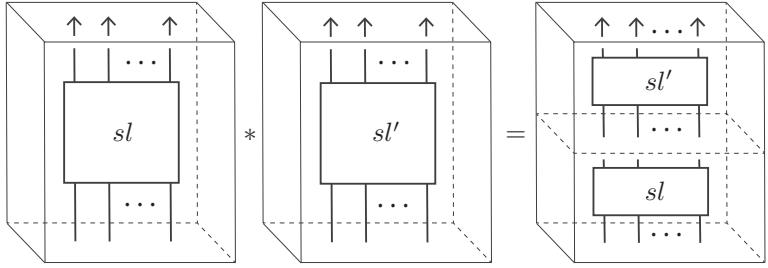}
     \put(40,46){$sl$}
     \put(137,46){$sl'$}
     \put(239,65.5){$sl'$}
     \put(239,24){$sl$}
     \put(89,46){$*$}
     \put(187,46){$=$}
  \end{overpic}
  \caption{A product of two string links $sl$ and $sl'$}
  \label{product}
 \end{center}
\end{figure}
\begin{figure}[htpb]
 \begin{center}
  \begin{overpic}[width=10cm]{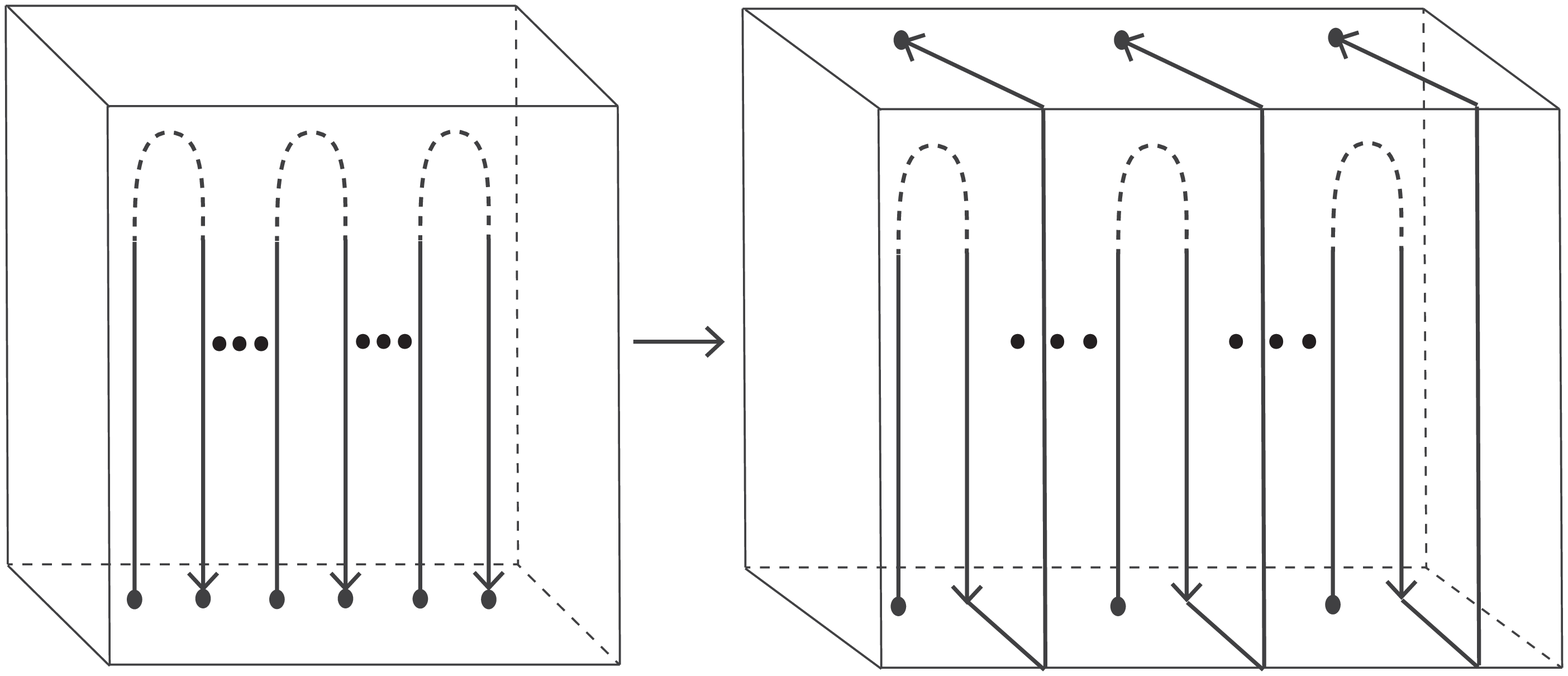}
     \put(5,-10){$\gamma=\gamma_{1}\cup\cdots\cup\gamma_{i}\cup\cdots\cup\gamma_{n}$}
     \put(22,5){$p_{1}$}
     \put(34,5){$q_{1}$}
     \put(48,5){$p_{i}$}
     \put(59,5){$q_{i}$}
     \put(73,5){$p_{n}$}
     \put(86,5){$q_{n}$}
     \put(162,5){$p_{1}$}
     \put(200,5){$p_{i}$}
     \put(239,5){$p_{n}$}
     \put(153,113){$q'_{1}$}
     \put(194,113){$q'_{i}$}
     \put(232,113){$q'_{n}$}
     \put(181,10){$b_{1}$}
     \put(221,10){$b_{i}$}
     \put(259,10){$b_{n}$}
     \put(191,50){$s_{1}$}
     \put(232,50){$s_{i}$}
     \put(271,50){$s_{n}$}
     \put(181,109){$t_{1}$}
     \put(220,109){$t_{i}$}
     \put(260,109){$t_{n}$}
  \end{overpic}
  \caption{A one-to-one correspondence between a string link and a bottom tangle}
  \label{bottomtangletostringlink}
 \end{center}
\end{figure}

An $n$-component {\it bottom tangle} $\gamma=\gamma_{1}\cup \gamma_{2}\cup \cdots \cup \gamma_{n}$ defined by Levine \cite{L} is a tangle with $\partial \gamma_{i}=\{ (\frac{2i-1}{2n+1},\frac{1}{2},0),(\frac{2i}{2n+1},\frac{1}{2},0) \}$ $\subset$ $\partial [0,1]^{3}$ 
 for each  $i\ (=1,2,\ldots,n)$.

We explain that there is a one-to-one correspondence between the sets of string links and bottom tangles in the following.
We first describe a construction of obtaining a string link from a bottom tangle $\gamma=\gamma_{1}\cup \gamma_{2}\cup \cdots \cup \gamma_{n}$.
For each $i\ (=1,2,\ldots, n)$, let $p_{i}$ and $q_{i}$ be the end points ($\frac{2i-1}{2n+1},\frac{1}{2},0$) and ($\frac{2i}{2n+1},\frac{1}{2},0$) of $\gamma_{i}$ respectively,
and let $q'_{i}$ be the point ($\frac{2i-1}{2n+1},\frac{1}{2},1$).
Let $b_{i}$ be a line segment between $q_{i}$ and ($\frac{2i}{2n+1},0,0$), $s_{i}$ a line segment between ($\frac{2i}{2n+1},0,0$) and ($\frac{2i}{2n+1},0,1$) and $t_{i}$ a line segment between ($\frac{2i}{2n+1},0,1$) and $q_{i}'$, see~Figure \ref{bottomtangletostringlink}.
By pushing the strands $\gamma_{i}\cup b_{i}\cup s_{i}\cup t_{i}$ into the interior of $[0,1]^{3}$ 
with fixing the end points $p_{i}$ and $q'_{i}$, we have a string link.

Conversely let $sl=sl_{1}\cup sl_{2}\cup \cdots \cup sl_{n}$ be a string link with $\partial sl_{i}=\{ p_{i},q'_{i}\}$ for each $i~(=1,2,\ldots,n)$.
By pushing the strands $sl_{i}\cup t_{i}\cup s_{i}\cup b_{i}$ into the interior of $[0,1]^{3}$ 
with fixing the end points $p_{i}$ and $q_{i}$, we have a bottom tangle.

\subsection{A bottom tangle obtained from a disk/band surface of a clover link}
Let $C_n$ be a graph consisting of $n$ oriented loops $e_{1},e_{2},\ldots ,e_{n}$, each loop $e_{i}$ connected to a vertex $v$ by an edge $f_{i}\ (i=1,2,\ldots,n)$ , see Figure~\ref{Cn-graph}.
An {\it $n$-clover link} in $S^3$ is a spatial graph of $C_n$ \cite{L}. 
The each part of a clover link corresponding to $e_{i}$, $f_{i}$ and $v$ of $C_n$ are 
called the {\it leaf}, {\it stem} and {\it root}, denoted by the same notations respectively.

\begin{figure}[htpb]
 \begin{center}
\vspace{-3mm}
  \begin{overpic}[width=5cm]{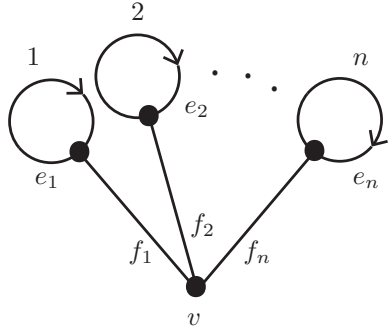}
     \put(7,108){1}
     \put(46,125){2}
     \put(129,108){$n$}
     \put(10,63){$e_{1}$}
     \put(66,90){$e_{2}$}
     \put(129,63){$e_{n}$}
     \put(45,35){$f_{1}$}
     \put(68,45){$f_{2}$}
     \put(88,35){$f_{n}$}
     \put(67,10){$v$}
  \end{overpic}
\vspace{-3mm}
  \caption{The graph $C_{n}$}
  \label{Cn-graph}
 \end{center}

\end{figure}

L. Kauffman, J. Simon, K. Wolcott and P. Zhao \cite{KSWZ} defined disk/band surfaces for spatial graphs.
For a spatial graph $\Gamma$, a {\it disk/band surface} $F_{\Gamma}$ of $\Gamma$ is a compact, oriented surface in $S^3$ such that $\Gamma$ is a deformation retract of $F_{\Gamma}$ contained in the interior of $F_{\Gamma}$. 
Note that any disk/band surface of a spatial graph is ambient isotopic to a surface constructed by putting a disk at each vertex of the spatial graph, connecting the disks with bands along the spatial edges.
We remark that for a spatial graph, there are infinitely many disk/band surfaces up to ambient isotopy.

Given an $n$-clover link, we construct an $n$-component bottom tangle using a disk/band surface of the clover link as follows.
\begin{itemize}
\item[(1)] For an $n$-clover link $c$, let $F_{c}$ be a disk/band surface of $c$ and let $D$ be a disk which contains the root.
From now on, we may assume that the intersection $D\cap \displaystyle\bigcup_{i=1}^{n}f_{i}$ and 
orientations of the disks are as illustrated in Figure~\ref{diskbandnotorikata}.

\item[(2)] Let $N(D)$ be the regular neighborhood of $D$.
\item[(3)] Since $S^{3}\setminus N(D)$ is homeomorphic to the $3$-ball, $F_{c}\setminus N(D)$ can be seen as a disjoint union of surfaces in the $3$-ball.
Hence $\partial F_{c}\setminus N(D)$ is a disjoint union of $n$-arcs and $n$-circles $\displaystyle\bigcup_{i=1}^{n}S_{i}^{1}$ in the $3$-ball.
\item[(4)] Since the $3$-ball is homeomorphic to $[0,1]^{3}$, we obtain an oriented ordered $n$-component bottom tangle $\gamma_{F_{c}}$ from $(\partial F_{c}\setminus N(D)) \setminus \displaystyle\bigcup_{i=1}^{n}S_{i}^{1}$ as illustrated in (3) and (4) of Figure~\ref{bottomtanglenokosei}.
We call $\gamma_{F_{c}}$ an {\it $n$-component  bottom tangle obtained from $F_{c}$}.
\end{itemize}
\begin{figure}[htpb]
 \begin{center}
  \begin{overpic}[width=10cm]{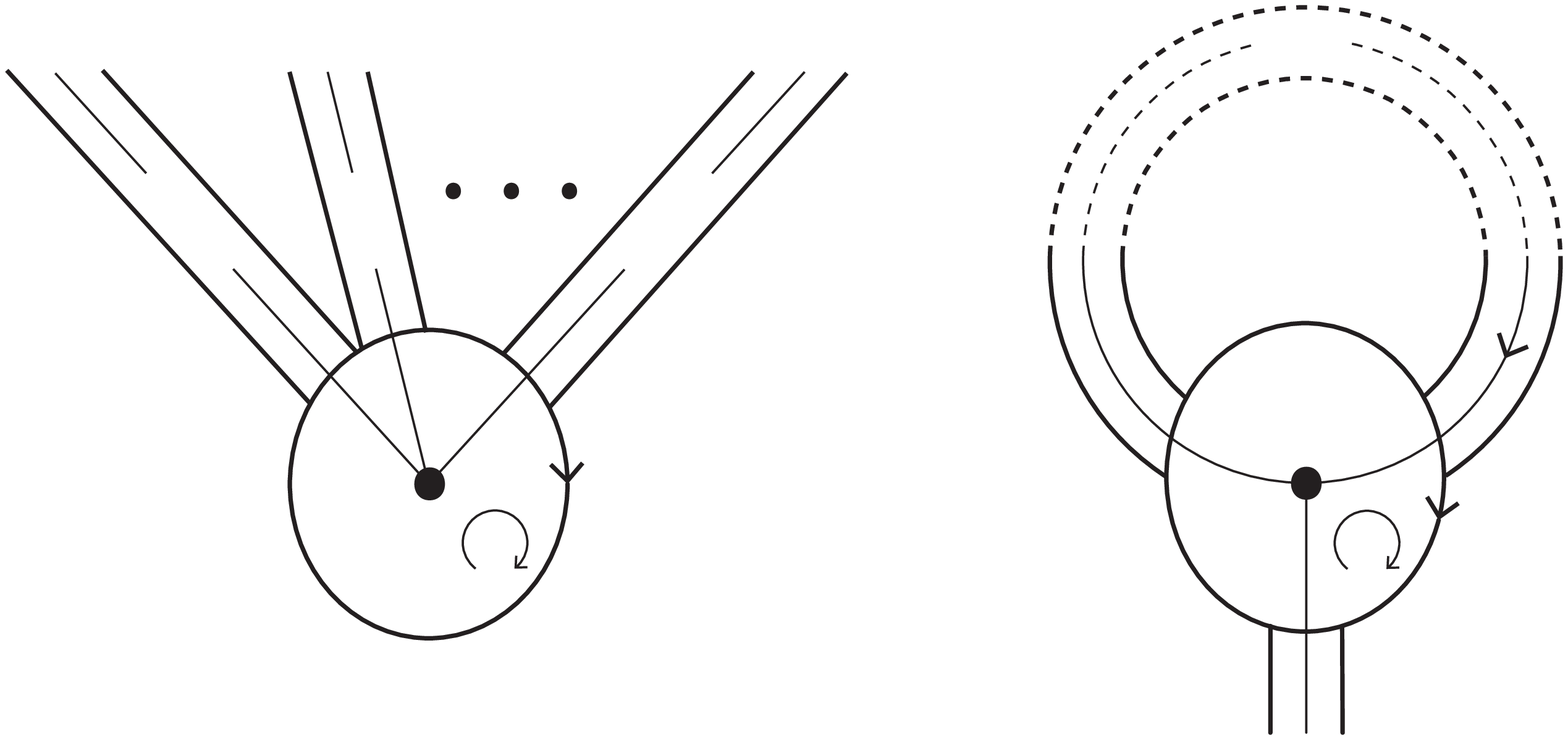}
     \put(32,90){$f_{1}$}
     \put(62,90){$f_{2}$}
     \put(114,90){$f_{n}$}
     \put(74,5){$D$}
     \put(234,125){$e_{i}$}
  \end{overpic}
  \caption{}
  \label{diskbandnotorikata}
 \end{center}
\end{figure}

\begin{figure}[htbp]
  \begin{center}
  \begin{overpic}[width=10cm]{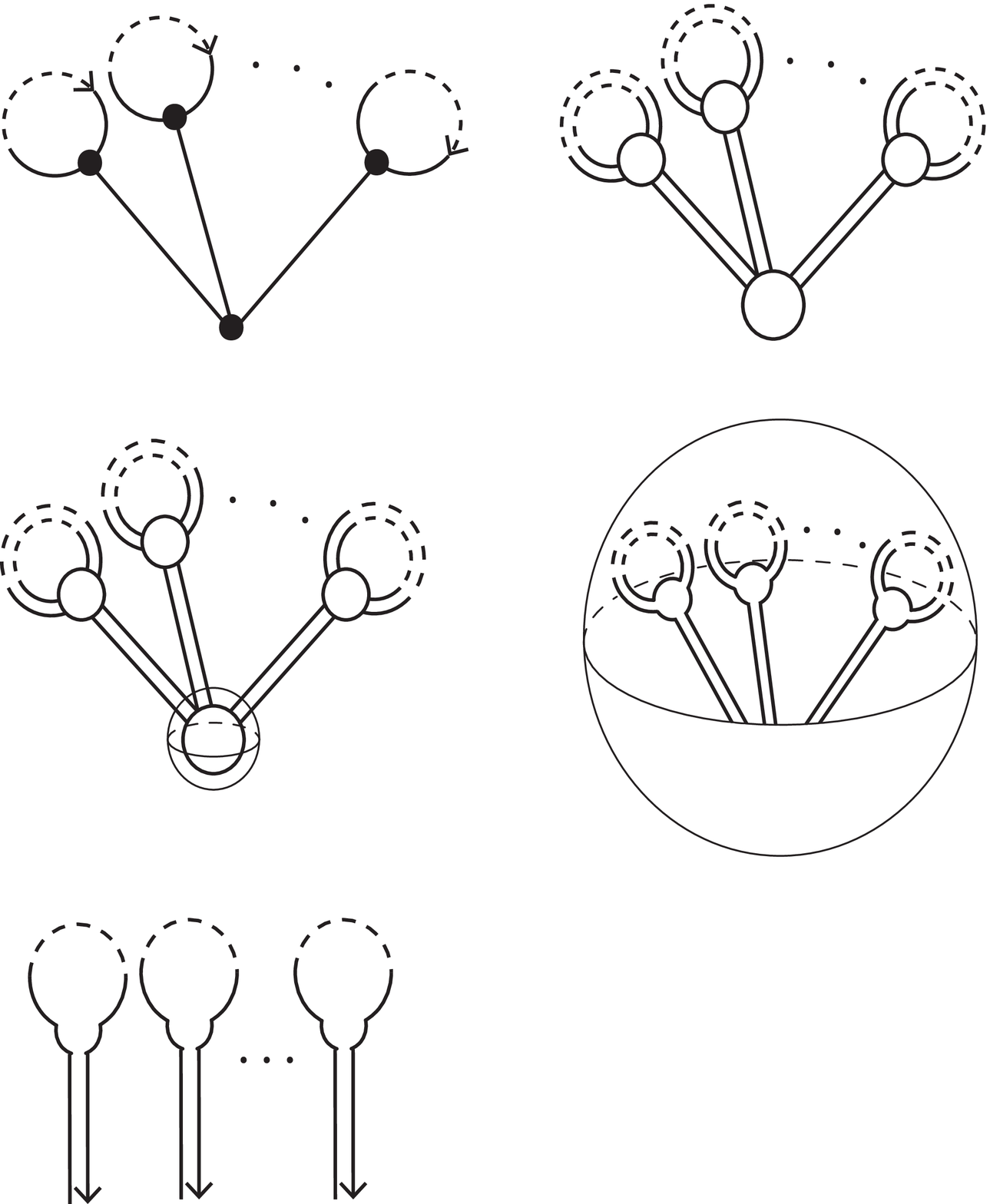}
   \put(1,325){1}
   \put(39,345){2}
   \put(124,328){$n$}
   \put(-25,285){$c$ :}
   \put(145,285){(1)}
   \put(235,253){$D$}
   \put(-25,155){(2)}
   \put(79,129){$N(D)$}
   \put(145,155){(3)}
   \put(200,90){$\partial F_{c}\setminus N(D)$}
   \put(-25,38){(4)}
   \put(20,85){1}
   \put(52,85){2}
   \put(96,85){$n$}
   \put(60,-5){$\gamma_{F_{c}}$}
\end{overpic}
    \caption{A method for obtaining a bottom tangle from a disk-band surface of a clover link}
    \label{bottomtanglenokosei}
  \end{center}
\end{figure}

\subsection{Milnor invariants}
\label{Milnor invariants}
Let us briefly recall from \cite{L} the definition of the Milnor number for a bottom tangle.
Let $\gamma=\gamma_{1}\cup\gamma_{2}\cup\cdots\cup\gamma_{n}$ be an oriented ordered $n$-component bottom tangle in $[0,1]^{3}$.
Let $G$ be the fundamental group of $[0,1]^{3}\setminus \gamma$ with a base point $p=(\frac{1}{2},0,0)$ and $G_{q}$ the $q$th lower central subgroup of $G$, namely $G_{1}=G$, $G_{q}$ is the subgroup generated by $\{ a^{-1}b^{-1}ab\ |\ a\in G, b\in G_{q-1}\}$.
Then the quotient group $G/G_{q}$ is generated by $\alpha_{1},\alpha_{2},\ldots,\alpha_{n}$~(\cite{Ch}, \cite{S}),
where $\alpha_{i}$ is the $i$th meridian of $\gamma$ which is represented by the composite path $t_{i}m_{i}{t_{i}}^{-1}$ in the $(x,y)$-plane, $m_{i}$ is a small counterclockwise circle about the point $p_{i}$ and $t_{i}$ is a straight line from $p_{i}$ to $m_{i}$, see Figure~\ref{meridian}.
Then the $i$th longitude $\lambda_{i}$ of $\gamma$ is represented by $\alpha_{1},\alpha_{2},\ldots,\alpha_{n}$ modulo $G_{q}$, where $\lambda_{i}$ is represented by the composite path $t_{i}l_{i}{t'_{i}}^{-1}$, $t'_{i}$ is a straight line from $p$ to the boundary of a small neighborhood of $q_{i}$ and $l_{i}$ is a path on the boundary of a small regular neighborhood of $\gamma_{i}$.
We assume that $\lambda_{i}$ is trivial in $G/G_{2}$.
See Figure~\ref{longitude}.
\begin{figure}[htbp]
\begin{minipage}{0.45\hsize}
  \begin{center}
    \begin{overpic}[width=5cm]{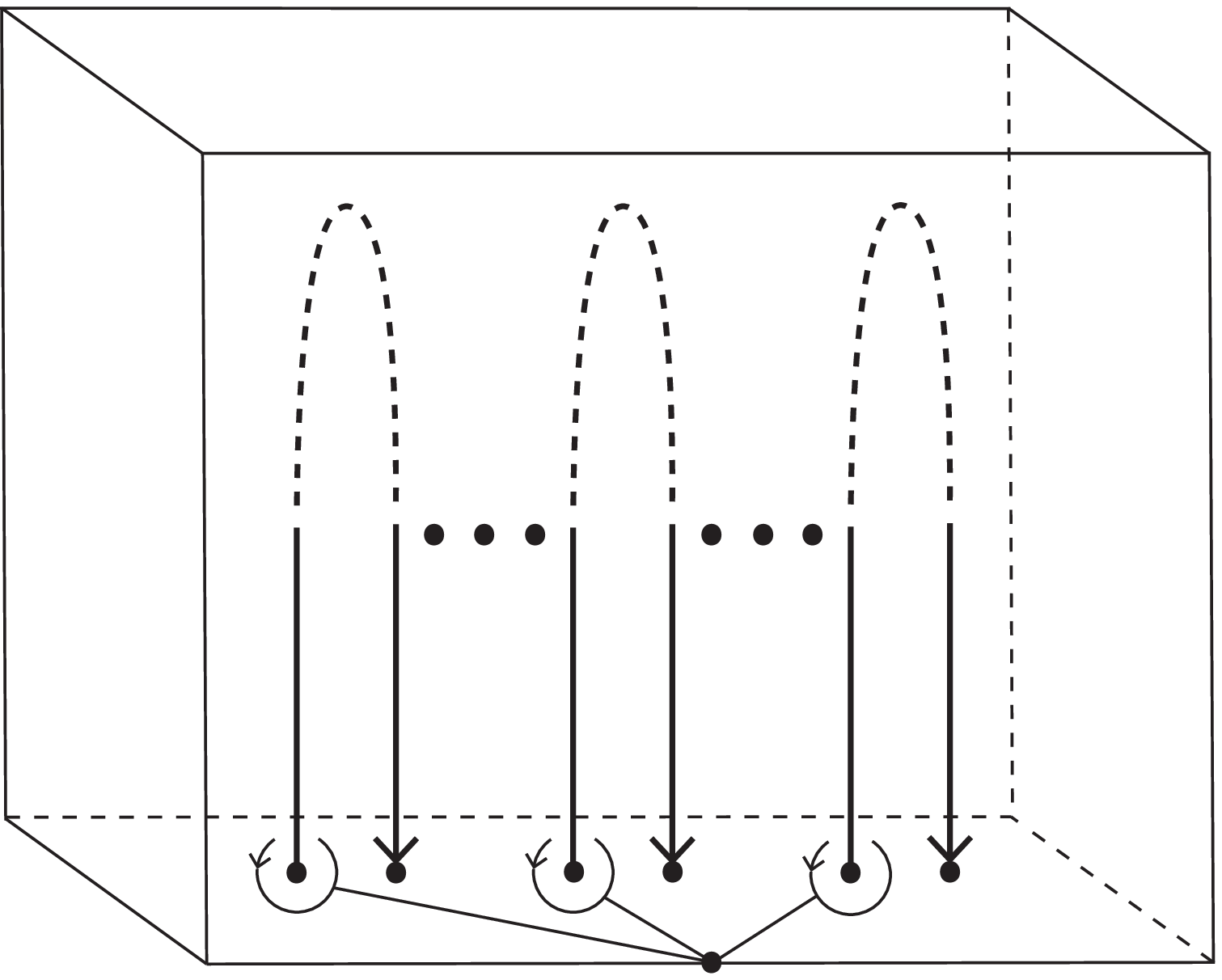}
       \put(26,10){$\alpha_{1}$}
       \put(52,18){$\alpha_{i}$}
       \put(102,10){$\alpha_{n}$}
       \put(81,0){$p$}
     \end{overpic}
  \end{center}
    \caption{meridians}
    \label{meridian}
\end{minipage}
\begin{minipage}{0.45\hsize}
 \begin{center}
    \begin{overpic}[width=5cm]{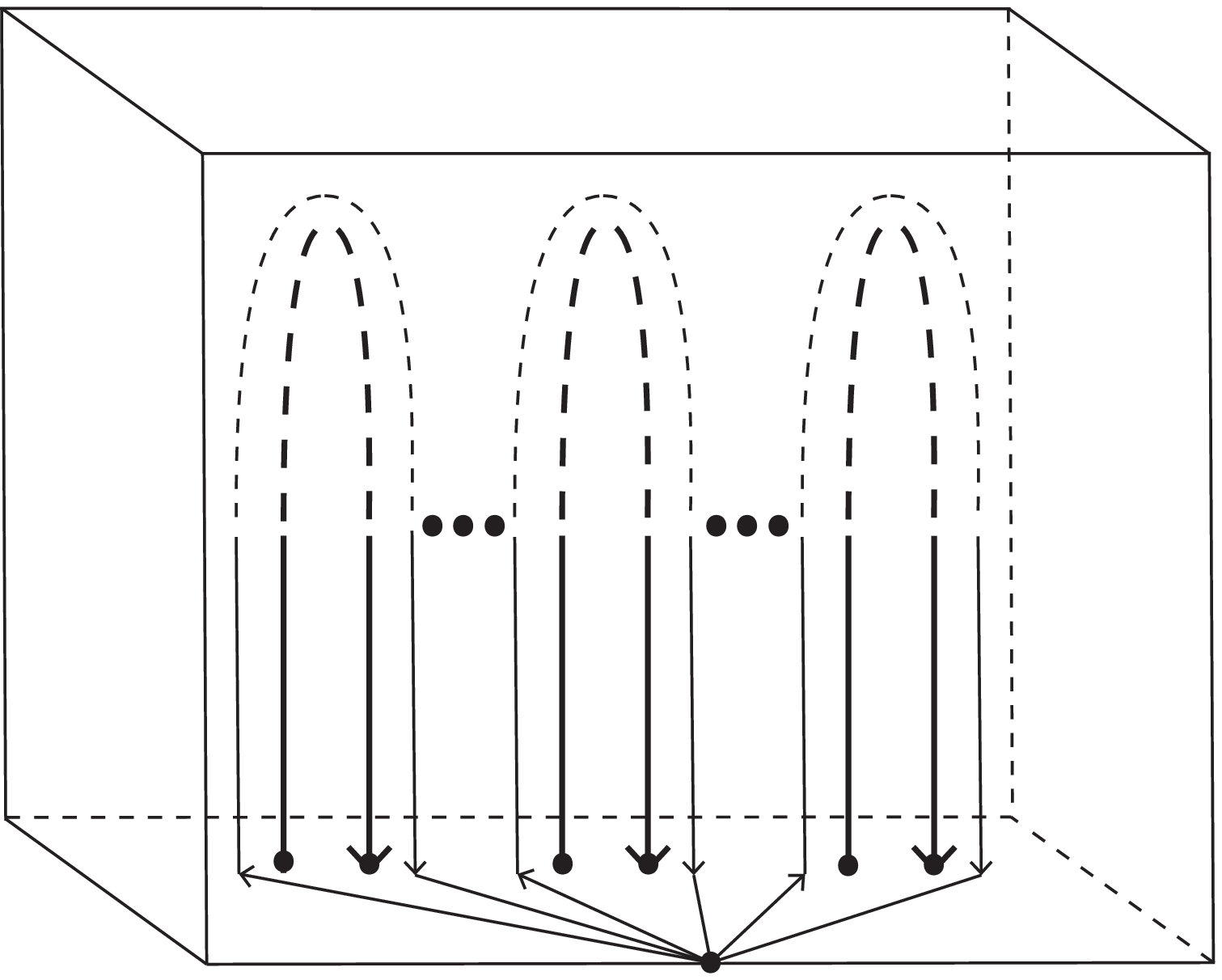}
       \put(25,10){$\lambda_{1}$}
       \put(51,28){$\lambda_{i}$}
       \put(112,10){$\lambda_{n}$}
       \put(81,0){$p$}
     \end{overpic}
  \end{center}
    \caption{longitudes}
    \label{longitude}
\end{minipage}
\end{figure}

We consider the Magnus expansion $E(\lambda_{j})$ of $\lambda_{j}$.
The Magnus expansion $E$ is a homomorphism from a free group $\langle \alpha_{1}, \alpha_{2}, \ldots , \alpha_{n}\rangle$ to the formal power series ring in non-commutative variables $X_{1},X_{2},\ldots,X_{n}$ with integer coefficients defined as follows.
$E(\alpha_{i})=1+X_{i},\ E(\alpha_{i}^{-1})=1-X_{i}+X_{i}^{2}-X_{i}^{3}+\cdots \ (i=1,2,\ldots , n)$.

For a sequence $I=i_{1}i_{2}\ldots i_{k-1}j\ (i_{m}\in \{1,2,\ldots,n\}, k\leq q)$, we define the {\it Milnor number} $\mu_{\gamma}(I)$ to be the coefficient of $X_{i_{1}}X_{i_{2}}\cdots X_{i_{k-1}}$ in $E(\lambda_{j})$ (we define $\mu_{\gamma}(j)=0$),
which is an invariant \cite{L}.
(In \cite{L}, the set of $\lambda_{j}$'s, without taking the Magnus expansion, is called the {\it Milnor $\omu$-invariant}.)
For a bottom tangle $\gamma=\gamma_{1}\cup\gamma_{2}\cup\cdots\cup\gamma_{n}$, an oriented ordered link $L(\gamma)=L_{1}\cup L_{2}\cup\cdots\cup L_{n}$ in $S^{3}$ can be defined by $L_{i}=\gamma_{i}\cup a_{i}$, where $a_{i}$ is a line segment connecting $p_{i}$ and $q_{i}$, see Figure~\ref{closure}.
\begin{figure}[htbp]
  \begin{center}
  \begin{overpic}[width=5cm]{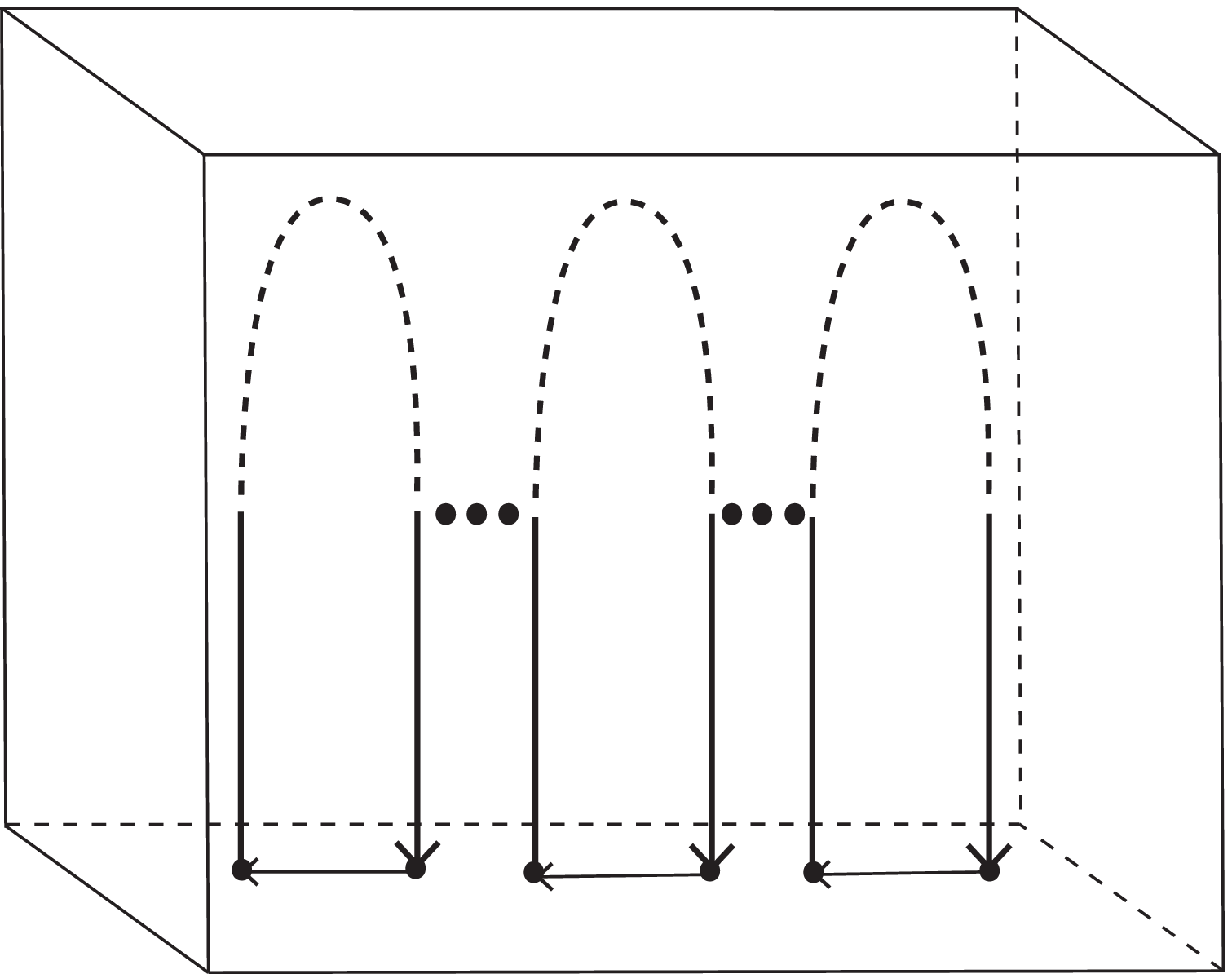}
   \put(35,5){$a_{1}$}
   \put(70,5){$a_{i}$}
   \put(100,5){$a_{n}$}
\end{overpic}
    \caption{}
    \label{closure}
  \end{center}
\end{figure}

We call $L(\gamma)$ the {\it closure} of $\gamma$.
On the other hand, for any link $L$ in $S^{3}$, there is a bottom tangle $\gamma_{L}$ such that the closure of $\gamma_{L}$ is equal to $L$.
So we define the Milnor number of $L$ to be the Milnor number of $\gamma_{L}$.
Let $\Delta_L(I)$ be the greatest common devisor of $\mu_L(J)'s$,
where $J$ is obtained from proper subsequence of $I$ by permuting cyclicly.
The {\it Milnor $\omu$-invariant $\omu_L(I)$} is the residue class of $\mu_L(I)$ modulo $\Delta_L(I)$.
We note that for a sequence $I$, if we have $\Delta_L(I)=0$, then the Milnor $\omu$-invariant $\omu_{L}(I)$ is equal to the Milnor number $\mu_{\gamma_{L}}(I)$.
Now we define the Milnor number for clover links.

\begin{definition}
Let $c$ be an $n$-clover link and $F_{c}$ a disk/band surface of $c$.
Let $\gamma_{F_{c}}$ be the $n$-component bottom tangle obtained from $F_{c}$.
For a sequence $I$, {\it the Milnor number $\mu_{c}(I)$ for $c$} is defined to be the Milnor number $\mu_{\gamma_{F_{c}}}(I)$.
\end{definition}

While $\mu_{c}(I)$ depends on a choice of $F_{c}$, we have the following result.

\begin{theorem}
\label{well-definedness}
Let $c$ be an $n$-clover link and $l_{c}$ a link which is the disjoint union of leaves of $c$.
If $\omu_{l_{c}}(J)=0$ for any sequence $J$ with $|J|\leq k$, then $\mu_{c}(I)$ is well-defined for any sequence $I$ with $|I|\leq 2k+1$.
\end{theorem}

\begin{remark}
Levine \cite{L} proved the same result as Theorem \ref{well-definedness} for {\em flat vertex} clover links under {\em flatly isotopy}.
See Subsection \ref{sectionwell-definedness} for the definition of a flatly isotopy.
\end{remark}
\subsection{Proof of Theorem \ref{well-definedness}}
\label{sectionwell-definedness}
In order to prove Theorem \ref{well-definedness}, we need two Lemmas \ref{lemma1} and \ref{lemma2}.

Two flat vertex graphs $\Gamma$ and $\Gamma'$ are {\it flatly isotopic} if there exists an isotopy $h_t:S^3 \rightarrow S^3 (t\in[0,1])$ such that $h_0=id$, $h_1(\Gamma)=\Gamma'$ and $h_t(\Gamma)$ is a flat vertex graph for each $t\in[0,1]$.
We call such a isotopy $h_{t}$ a {\it flatly isotopy} \cite{Yamada}.

An $n$-component {\it braid} $\beta=\beta_{1}\cup\beta_{2}\cup\cdots\cup\beta_{n}$ is a tangle in $[0,1]^{3}$ such that for each $t\ (0\leq t\leq 1)$, $\displaystyle\bigcup_{i=1}^{n}\beta_{i}$ intersects $[0,1]\times[0,1]\times \{t\}$ transversely at $n$ points.
In particular, $\beta$ is a {\it pure braid} if for any $i\ (=1,2,\ldots,n)$, the boundary $\partial \beta_{i}=\{ (x,\frac{1}{2},0),(x,\frac{1}{2},1) \} \subset \partial [0,1]^{3}$ for $x\in [0,1]$.

\begin{lemma}
\label{lemma1}
If two $n$-clover links $c$ and $c'$ are ambient isotopic to each other, 
then there exists an $n$-clover link $c''$ such that  $c''$ is obtained from $c$ by a single B-move and $c''$ is flatly isotopic to $c'$,
where a B-move is a local move as illustrated in Figure~\ref{B-move}.
\end{lemma}
\begin{figure}[htbp]
  \begin{center}
    \begin{overpic}[width=10cm]{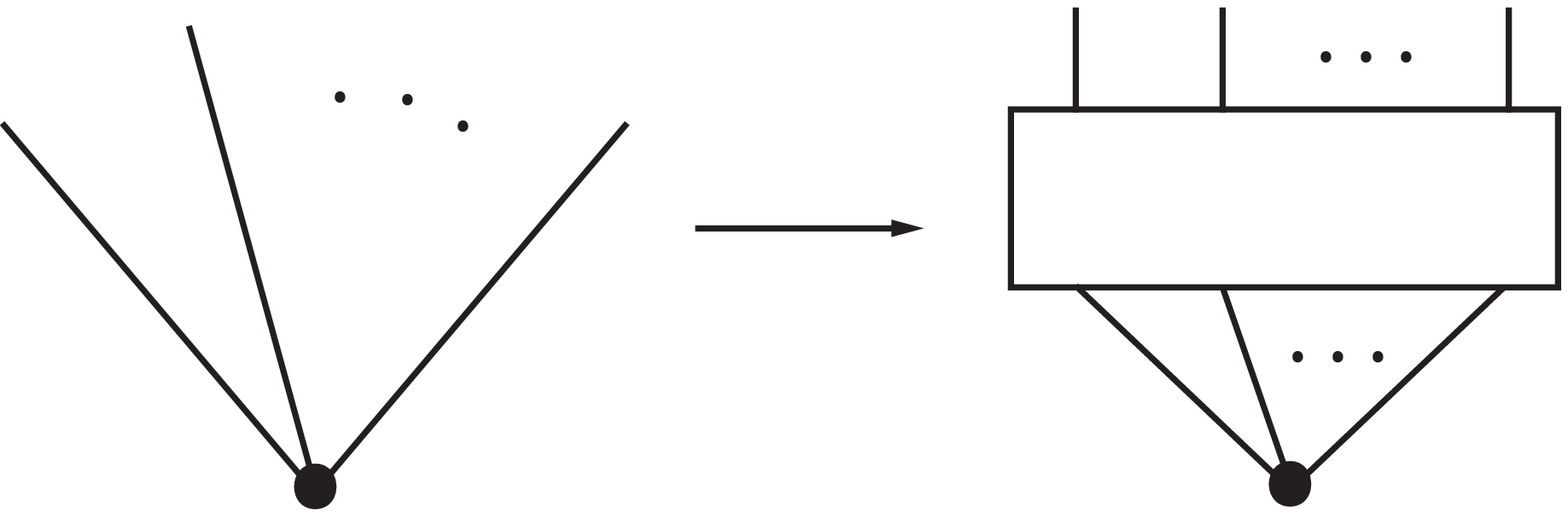}
    \linethickness{3pt}
    \put(130,77){B-move}
    \put(223,75){braid}
    \end{overpic}
    \caption{A B-move}
    \label{B-move}
  \end{center}
\end{figure}

\begin{proof}
Consider regular diagrams $\widetilde{c}, \widetilde{c'}$ of $c, c'$ respectively.
Since $c$ and $c'$ are ambient isotopic to each other, $\widetilde{c}$ and $\widetilde{c'}$ are related by a finite sequence of Reidemeister moves (i) $\sim$ (v) as illustrated in Figure~\ref{Reidemeister-moves} \cite{K}.
\begin{figure}[htbp]
  \begin{center}
    \begin{overpic}[width=8.5cm]{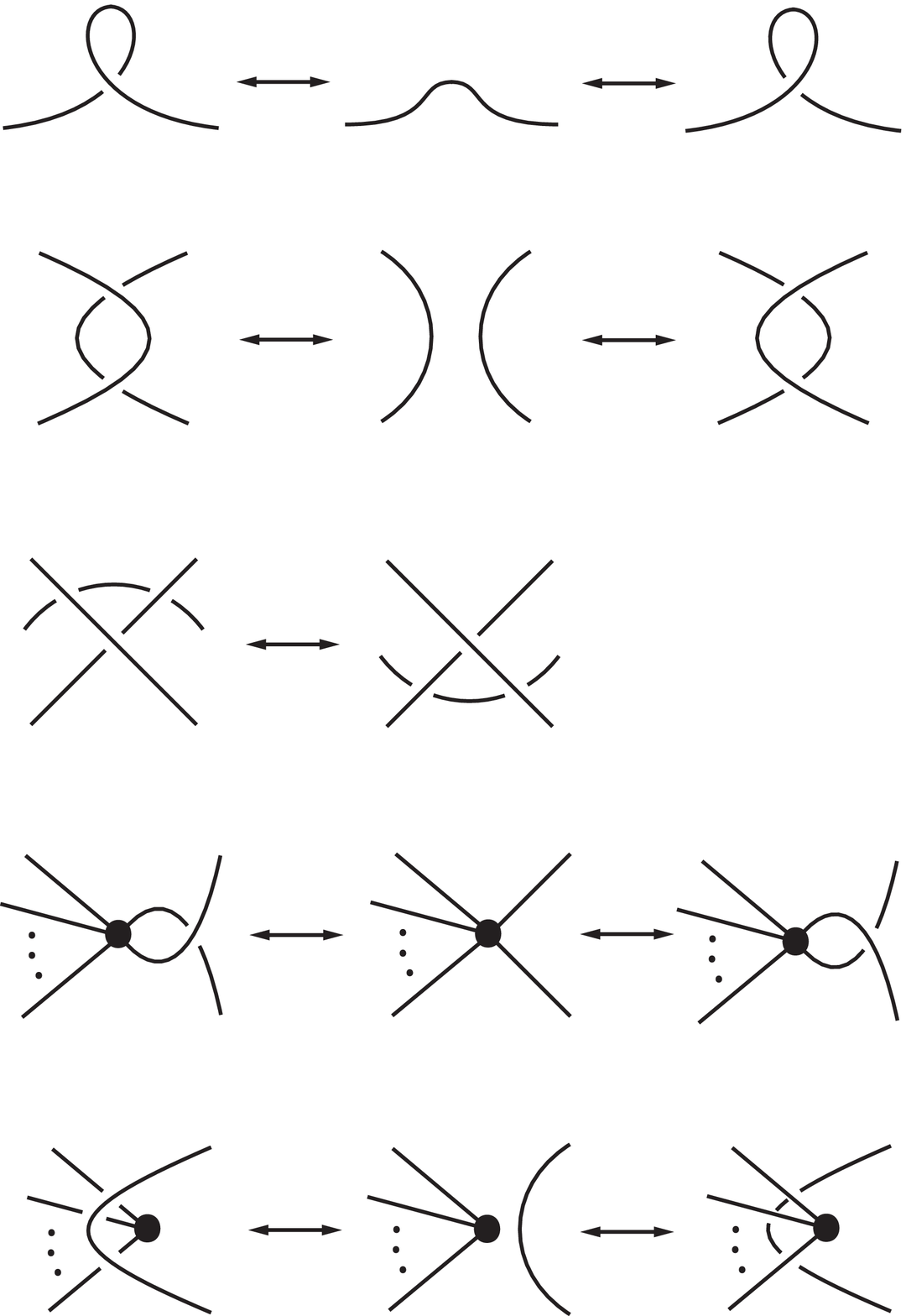}
    \put(-20,330){(i)}
    \put(-20,260){(ii)}
    \put(-20,178){(iii)}
    \put(-20,100){(iv)}
    \put(-20,22){(v)}
    \end{overpic}
    \caption{Reidemeister moves (i)$\sim$(v)}
    \label{Reidemeister-moves}
  \end{center}
\end{figure}

Moves excepting move (iv) are realized by flatly isotopies.
Move (iv) consists of two kinds transformations.
One increases a crossing, and the other decreases a crossing.
We call those two kinds of transformations move (iv)$^{+}$ and move (iv)$^{-}$ respectively.
It is not hard to see that a move (iv)$^{-}$ is realized by moves (iv)$^{+}$ and (ii).
Hence moves (i) $\sim$ (v) are realized by moves (i) $\sim$ (iii), (v) and (iv)$^{+}$.

We consider a move (iv)$^{+}$.
If we apply a move (iv)$^{+}$ to a regular diagram $\widetilde{d}$ 
of an $n$-clover link, then we obtain a regular diagram either (a) or (b) as illustrated in Figure~\ref{applying-move(iv)}.
The diagram (a) is clearly obtained from $\widetilde{d}$ by a single $\widetilde{B}$-move,
where $\widetilde{B}$-move is a local move of regular diagrams which correspond to B-move.
The diagram (b) is transformed into a diagram (b$'$) as illustrated in Figure~\ref{about(b)} 
by a finite sequence of moves (i)$\sim$(iii) and (v).
So the diagram (b) is obtained from $\widetilde{d}$ by moves (i)$\sim$(iii), (v) and a 
$\widetilde{B}$-moves.
Thus a move (iv)$^{+}$ is realized by moves (i)$\sim$(iii), (v) and a $\widetilde{B}$-move.
It follows that moves (i) $\sim$ (v) are realized by moves (i) $\sim$ (iii), (v) and $\widetilde{B}$-moves.
\begin{figure}[htbp]
  \begin{center}
    \begin{overpic}[width=10cm]{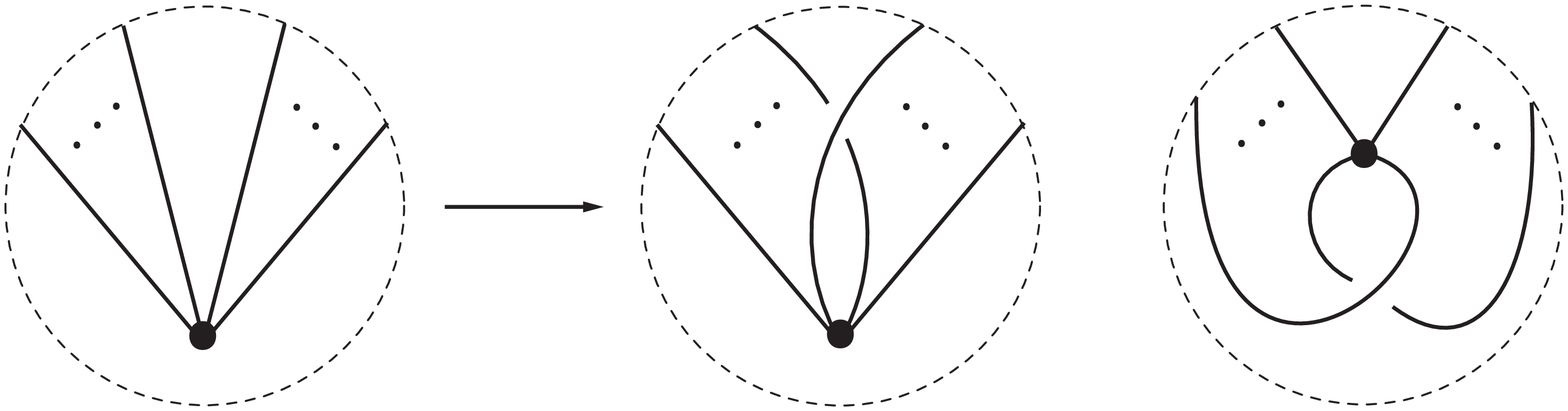}
    \put(83,44){(iv)$^{+}$}
    \put(200,37){or}
    \put(28,-12){$\widetilde{d}$}
    \put(147,-10){(a)}
    \put(246,-10){(b)}
    \put(-12,57){$f_{1}$}
    \put(10,77){$f_{i}$}
    \put(52,77){$f_{i+1}$}
    \put(69,57){$f_{n}$}
    \put(108,57){$f_{1}$}
    \put(125,77){$f_{i+1}$}
    \put(169,77){$f_{i}$}
    \put(190,57){$f_{n}$}
    \put(208,62){$f_{1}$}
    \put(229,77){$f_{i}$}
    \put(265,77){$f_{i+1}$}
    \put(283,62){$f_{n}$}
    \end{overpic}
    \caption{Regular diagrams obtained by applying a move (iv)$^{+}$ to a regular diagram $\widetilde{d}$}
    \label{applying-move(iv)}
  \end{center}
\end{figure}
\begin{figure}[htbp]
  \begin{center}
    \begin{overpic}[width=7.5cm]{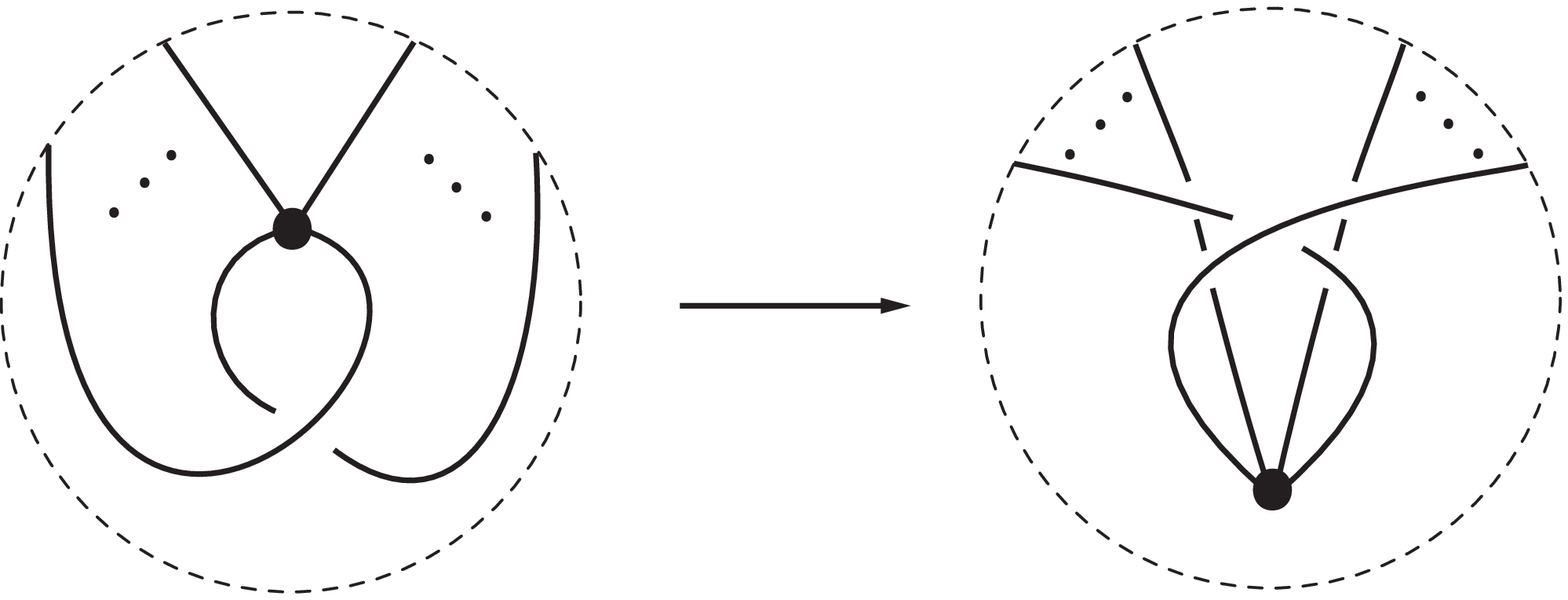}
    \put(83,47){(i)'s$\sim$(iii)'s}
    \put(90,32){and (v)'s}
    \put(36,-10){(b)}
    \put(166,-10){(b$'$)}
    \put(-3,68){$f_{1}$}
    \put(19,84){$f_{i}$}
    \put(52,84){$f_{i+1}$}
    \put(73,68){$f_{n}$}
    \put(126,62){$f_{1}$}
    \put(151,84){$f_{i}$}
    \put(187,84){$f_{i+1}$}
    \put(212,62){$f_{n}$}
    \end{overpic}
    \caption{}
    \label{about(b)}
  \end{center}
\end{figure}

Applying a move (v), subsequently a $\widetilde{B}$-move can be realized by applying a $\widetilde{B}$-move, subsequently a move (v) and several moves (ii) and (iii), see Figure~\ref{commutative}.
It is clear that applying each move ($*$)~($*=$ i, ii, iii), subsequently a $\widetilde{B}$-move can be realized by applying $\widetilde{B}$-moves, subsequently a move ($*$).
\begin{figure}[htbp]
  \begin{center}
    \begin{overpic}[width=10cm]{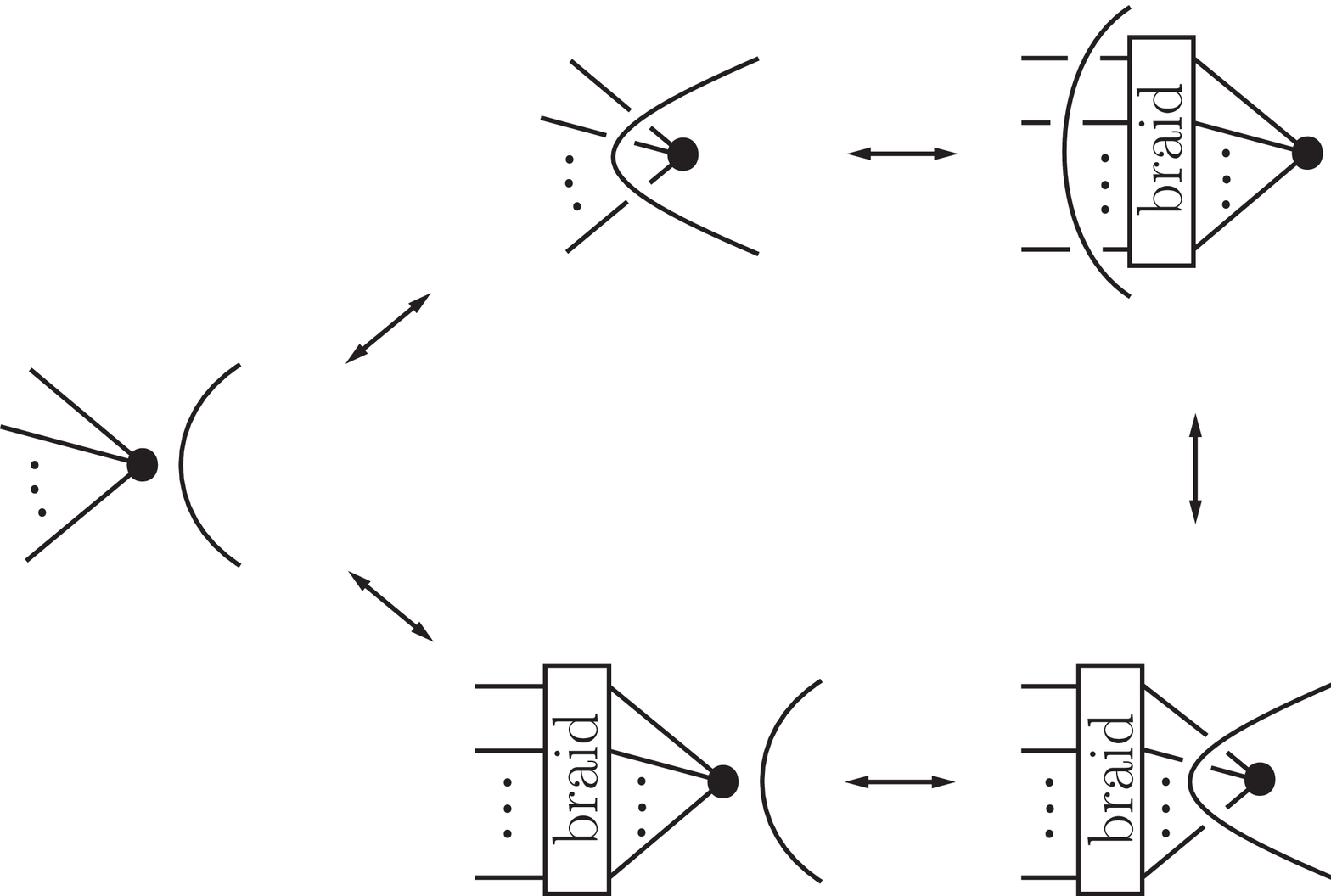}
    \put(47,55){$\widetilde{B}$-move}
    \put(72,127){(v)}
    \put(175,165){$\widetilde{B}$-move}
    \put(185,30){(v)}
    \put(260,88){(ii)'s and (iii)'s}
    \end{overpic}
    \caption{}
    \label{commutative}
  \end{center}
\end{figure}
So applying several moves (i)$\sim$(iii) and (v), subsequently a $\widetilde{B}$-move
can be realized by applying a $\widetilde{B}$-move, subsequently several moves (i)$\sim$(iii) and (v).

Therefore there exists a regular diagram $\widetilde{c''}$ of an $n$-clover link $c''$ such that 
$\widetilde{c''}$ is obtained from $\widetilde{c}$ by a finite sequence of $\widetilde{B}$-moves and $\widetilde{c'}$ is obtained from $\widetilde{c''}$ by a finite sequence of moves (i)$\sim$(iii) and (v).
We note that a finite sequence of $\widetilde{B}$-moves is realized by a single $\widetilde{B}$-move.
This completes the proof.
\end{proof}

Since we suppose that the disk parts of any disk/band surfaces for a clover link  
are as illustrated in Figure~\ref{diskbandnotorikata}, by Lemma \ref{lemma1}, we immediately have the following Proposition~\ref{propPB-move}.
\begin{proposition}
\label{propPB-move}
For an $n$-clover link $c$, any two disk/band surfaces $F_{c}$ and $F'_{c}$ are transformed into each other by adding full-twists to bands {\rm (}Figure~{\rm \ref{PB-move} (a))} and a single move illustrated in 
Figure~{\rm \ref{PB-move} (b)}.
\end{proposition}
\begin{figure}[htpb]
 \begin{center}
  \begin{overpic}[width=10cm]{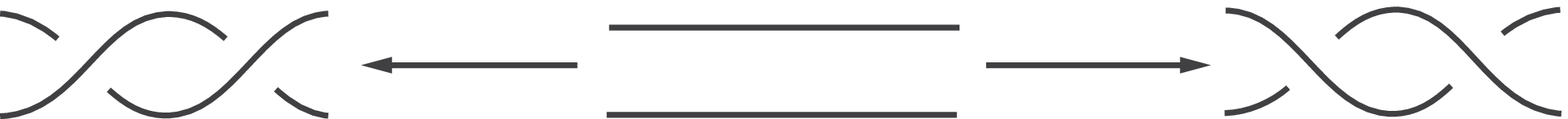}
     \put(-20,8){(a)}
     \put(130,6){band}
     \put(67,15){full-twist}
     \put(180,15){full-twist}
  \end{overpic}
 \end{center}
\end{figure}

\begin{figure}[htbp]
  \begin{center}
    \begin{overpic}[width=10cm]{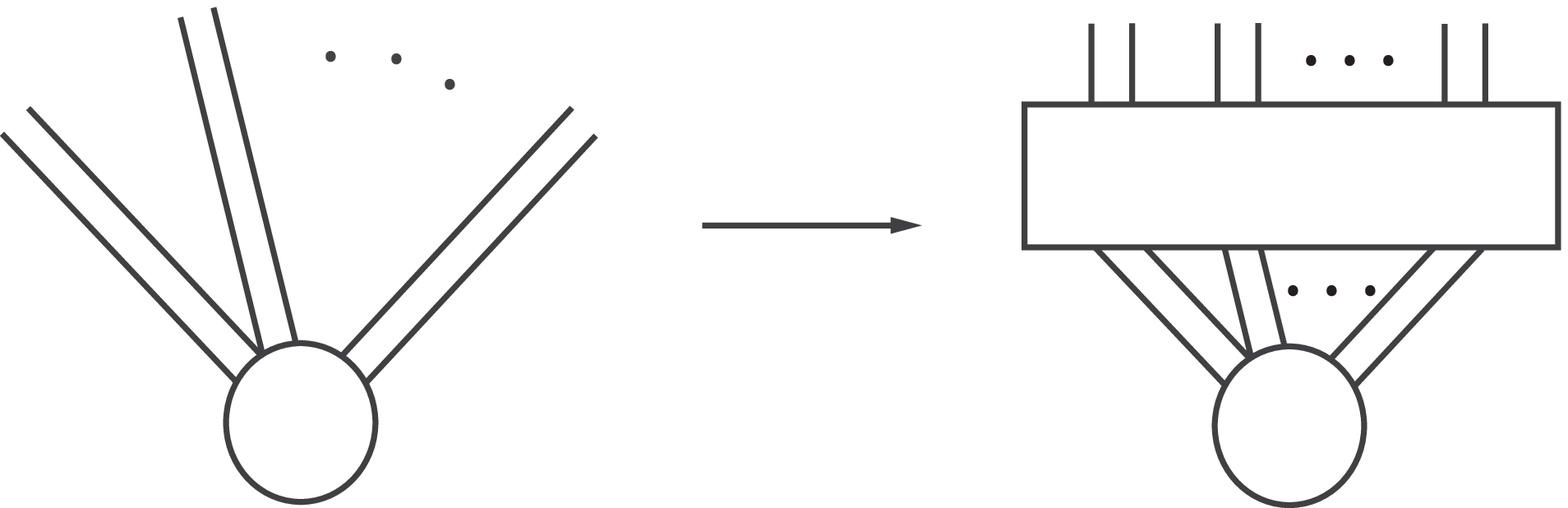}
    \put(-20,46){(b)}
    \put(213,58){pure braid}
    \end{overpic}
    \caption{Two local moves of disk-band surfaces}
    \label{PB-move}
  \end{center}
\end{figure}

\begin{figure}[htbp]
    \begin{overpic}[width=6cm]{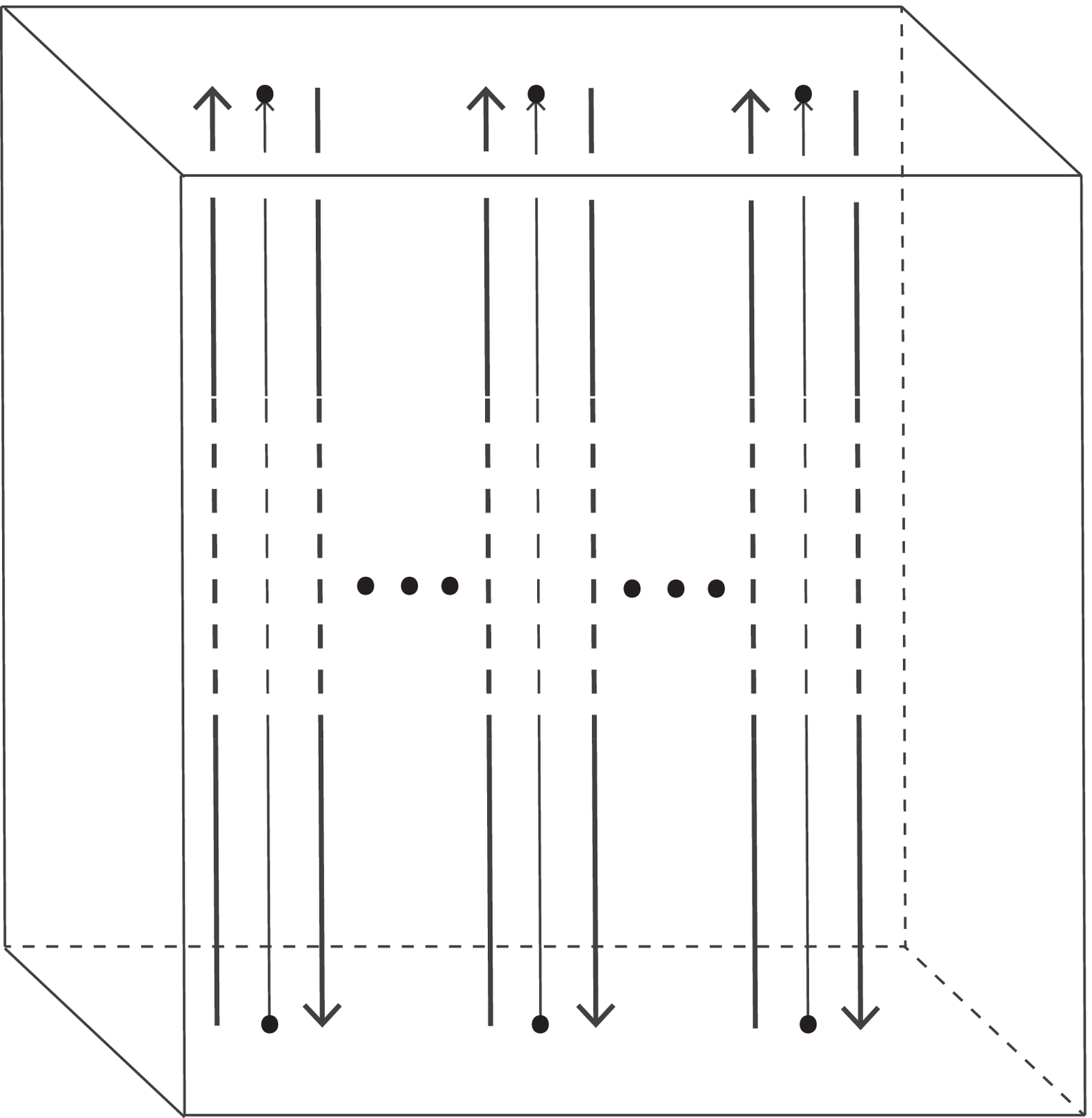}
    \put(18,167){$u_{11}$}
    \put(33,167){$u_{1}$}
    \put(44,167){$u_{12}$}
    \put(63,167){$u_{i1}$}
    \put(77,167){$u_{i}$}
    \put(87,167){$u_{i2}$}
    \put(103,167){$u_{n1}$}
    \put(119,167){$u_{n}$}
    \put(131,167){$u_{n2}$}
    \end{overpic}
    \caption{}
    \label{parallelstrands}
\end{figure}

Here we introduce a SL-move which is a transformation of an $n$-component bottom tangle $\gamma=\gamma_{1}\cup\gamma_{2}\cup\cdots\cup\gamma_{n}$ in $[0,1]^{3}$
\begin{itemize}
\item[(1)]
Let $u=u_{1}\cup u_{2}\cup\cdots\cup u_{n}$ be an oriented ordered $n$-component string link in $[0,1]^{3}$.
For each $i\ (=1,2,\ldots,n)$, we consider two arcs $u_{i1}$ and $u_{i2}$ which are parallel to the $i$th component $u_{i}$ of $u$ with orientations as illustrated in Figure~\ref{parallelstrands}.
Let $u'=(u_{11}\cup u_{12})\cup \cdots \cup (u_{n1}\cup u_{n2})$.
We may assume that for each $i~(=1,2,\ldots,n)$, $\partial u_{i1}=\{ (\frac{2i-1}{2n+1},\frac{1}{2},0), (\frac{2i-1}{2n+1},\frac{1}{2},1)\}$ and $\partial u_{i2}=\{ (\frac{2i}{2n+1},\frac{1}{2},0), (\frac{2i}{2n+1},\frac{1}{2},1) \}$.

\item[(2)] Let $\gamma'=\gamma'_{1}\cup\gamma'_{2}\cup\cdots\cup\gamma'_{n}$ be an $n$-component bottom tangle in $[0,1]^{3}$ defined by 
\begin{center}
$\gamma'_{i}=h_{0}(u_{i1}\cup u_{i2})\cup h_{1}(\gamma_{i})$
\end{center}
for $i=1,2,\ldots,n$, where $h_{0},h_{1}:([0,1]\times[0,1])\times[0,1]\rightarrow ([0,1]\times[0,1])\times[0,1]$ are embeddings defined by 
\begin{center}
$h_{0}(x,t)=(x,\frac{1}{2}t)$ and $h_{1}(x,t)=(x,\frac{1}{2}+\frac{1}{2}t)$
\end{center}
for $x\in([0,1]\times[0,1])$ and $t\in[0,1]$.
\end{itemize}
We say that {\it $\gamma'$ is obtained from $\gamma$ by a SL-move}.
We note that if $u$ is trivial, a SL-move is just adding full-twists or nothing.
A SL-move is determined by a {\em String Link} and a number of full-twists; this explains \lq SL\rq  in SL-move.
For example, see Figure~\ref{SL-move}

\begin{figure}[htbp]
  \begin{center}
    \begin{overpic}[width=8cm]{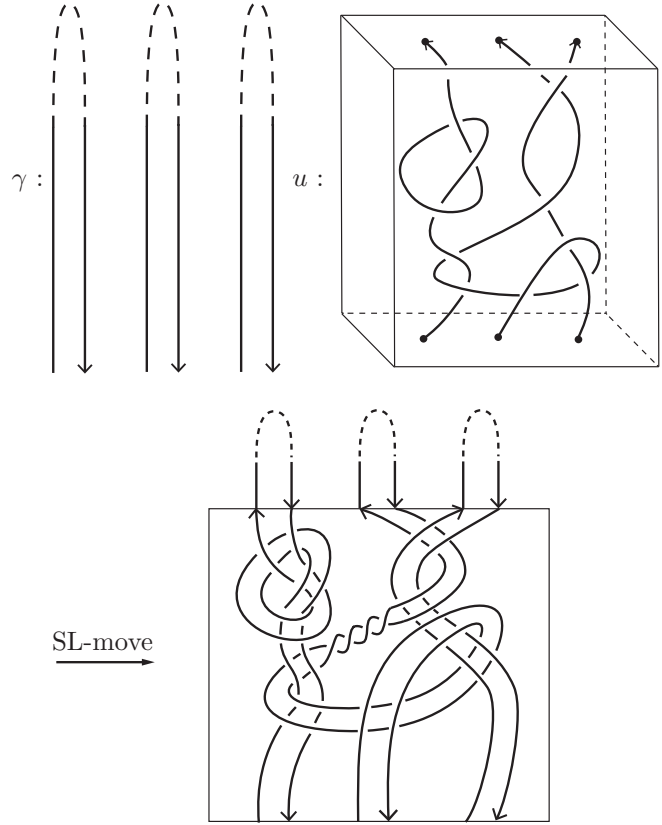}
    \put(-15,240){$\gamma$ :}
    \put(90,240){$u$ :}
    \put(0,65){SL-move}
    \end{overpic}
    \caption{An example of a SL-move}
    \label{SL-move}
  \end{center}
\end{figure}

The following lemma, which is not directly mentioned, can be proved by combining some results in \cite{L}.
In the following we directly prove by calculus of Milnor numbers.
\begin{lemma}\cite{L}
\label{lemma2}
Let $\gamma$ be an $n$-component bottom tangle and $\gamma'$ a bottom tangle obtained from $\gamma$ by a SL-move.
If $\mu_{\gamma}(J)=\mu_{\gamma'}(J)=0$ for any sequence $J$ with $|J|\leq k$, then $\mu_{\gamma}(I)=\mu_{\gamma'}(I)$ for any sequence $I$ with $|I|\leq 2k+1$.
\end{lemma}

\begin{proof}
Denote respectively by $\alpha_i, \lambda_i$ (resp. $\alpha_i', \lambda_i'$) the $i$th meridian and $i$th longitude of $\gamma$ (resp. $\gamma'$) for $1\leq i\leq n$.
Let $E_X$ (resp. $E_Y$) be the Magnus expansion in non-commutative variables $X_1,\ldots ,X_n$ (resp. $Y_1,\ldots ,Y_n$) obtained by replacing $\alpha_j$ by $1+X_j$ (resp. $\alpha_j'$ by $1+Y_j$) and replacing $\alpha_j^{-1}$ by $1-X_j+X_j^2-X_j^3+\cdots$(resp. $\alpha_j'^{-1}$ by $1-Y_j+Y_j^2-Y_j^3+\cdots$) for $1\leq j \leq n$.
Let $u_{i}$ be the $i$th longitude of a string link which gives the SL-move, see Figure~\ref{bandbottomtangle}.

\begin{figure}[htbp]
  \begin{center}
    \begin{overpic}[width=10cm]{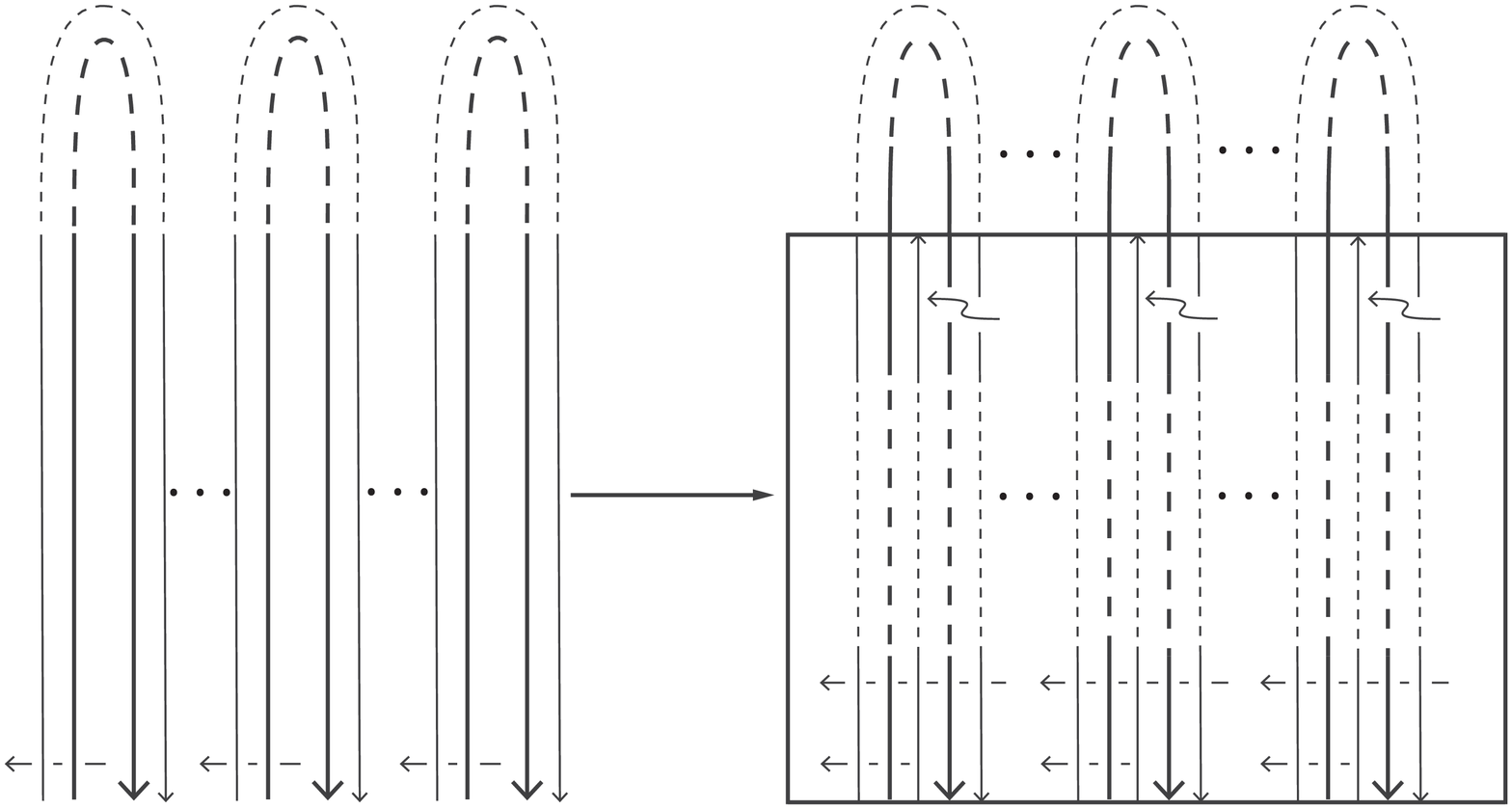}
    \linethickness{3pt}
    \put(53,-10){$\gamma$}
    \put(211,-10){$\gamma'$}
    \put(107,65){SL-move}
    \put(15,158){$\lambda_{1}$}
    \put(52,158){$\lambda_{i}$}
    \put(88,158){$\lambda_{n}$}
    \put(168,158){$\lambda'_{1}$}
    \put(211,158){$\lambda'_{i}$}
    \put(251,158){$\lambda'_{n}$}
    \put(-5,13){$\alpha_{1}$}
    \put(33,13){$\alpha_{i}$}
    \put(69,13){$\alpha_{n}$}
    \put(150,13){$\alpha'_{1}$}
    \put(192,13){$\alpha'_{i}$}
    \put(232,13){$\alpha'_{n}$}
    \put(190,90){$u_{1}$}
    \put(230,90){$u_{i}$}
    \put(272,90){$u_{n}$}
    \put(150,30){$\beta_{1}$}
    \put(192,30){$\beta_{i}$}
    \put(232,30){$\beta_{n}$}
    \end{overpic}
    \caption{}
    \label{bandbottomtangle}
  \end{center}
\end{figure}

Then we have $\alpha_i=u_i^{-1}\alpha_i'u_i$ and $\lambda_i'=u_i\lambda_iu_i^{-1}$, where $\alpha_{i}$ and $\lambda_{i}$ are assumed to be elements of $\pi_{1}([0,1]^{3}\setminus \gamma')$.
Let $\beta_{i}=[ \lambda_i', \alpha_i']$ as illustrated in Figure~\ref{bandbottomtangle}.
By the assumption,
Milnor numbers for $\gamma$ and $\gamma'$ of length $\leq k$ vanish,
so $E_X(\lambda_i)$ and $E_Y(\lambda'_i)$ can be written respectively in the form
\[
E_X(\lambda_i)=1+f_i(X)
~\text{and}~
E_Y(\lambda'_i)=1+f'_i(Y),
\]
where $f_i(X)$ and $f'_i(Y)$ mean the terms of degree $\geq k$.

First, we observe $E_Y(\alpha_i)$.
Let $E_Y({\lambda'_i}^{-1})=1+\overline{f'_i}(Y)$,
where $\overline{f'_i}(Y)$ is the terms of degree $\geq k$.
Note that
\[E_Y({\lambda'_i}^{-1})E_Y(\lambda'_i)=(1+\overline{f'_i}(Y))(1+f'_i(Y))=1.\]
It follows that 
\[\begin{array}{rcl}
E_Y(\beta_i)&=&E_Y([ \lambda_i', \alpha_i'])\\
&=&E_Y({\lambda'_i}^{-1}{\alpha'_i}^{-1}\lambda'_i\alpha'_i)\\
&=&(1+\overline{f'_i}(Y))E_Y({\alpha'_i}^{-1})(1+f'_i(Y))E_Y(\alpha'_i)\\
&=&(1+\overline{f'_i}(Y))(1+f'_i(Y)+\mathcal{O}(k+1))\\
&=&1+\mathcal{O}(k+1),
\end{array}\]
where $\mathcal{O}(m)$ means the terms of degree $\geq m$.
Since $u_i$ is represented by a product of the conjugates of $\beta_{i}$,
we have
$E_Y(u_i)=1+\mathcal{O}(k+1)$
and set 
\[E_Y(u_i)=1+g_i(Y)~\text{and}~E_Y(u_i^{-1})=1+\overline{g_i}(Y),\]
where $g_i(Y)$ and $\overline{g_i}(Y)$ mean the terms of degree $\geq k+1$.
As $\alpha_i=u_i^{-1}\alpha'_iu_i$
we have
\[\begin{array}{rcl}
E_Y(\alpha_i)&=&E_Y(u_i^{-1}\alpha'_iu_i)\\
&=&(1+\overline{g_i}(Y))(1+Y_i)(1+g_i(Y))\\
&=&1+(1+\overline{g_i}(Y))Y_i(1+g_i(Y))\\
&=&1+Y_i+\mathcal{O}(k+2).
\end{array}\]
Hence \[E_Y(\alpha_i)=E_Y(\alpha_i')+\mathcal{O}(k+2).\]

Now we observe $E_Y(\lambda'_{i})-(1+f_i(Y))$.
Since $\lambda_{i}$ is represented by a word of $\alpha_{1}^{\pm1},\alpha_{2}^{\pm1},\ldots,\alpha_{n}^{\pm1}$, then $E_Y(\lambda_{i})$ is obtained from the word by substituting $\alpha_{i}^{\pm1}$ for $E_Y(\alpha_{i}^{\pm1})=E_Y({\alpha'_i}^{\pm1})+\mathcal{O}(k+2)$.
Therefore
\[E_Y(\lambda_{i})-(1+f_{i}(Y))=\mathcal{O}(k+(k+2)-1)
=\mathcal{O}(2k+1).\]
It follows that
\[\begin{array}{rcl}
E_Y(\lambda_i')&=&E_Y(u_i\lambda_iu_i^{-1})\\
&=&(1+g_{i}(Y))(1+f_{i}(Y)+\mathcal{O}(2k+1))(1+\overline{g_{i}}(Y))\\
&=&1+f_{i}(Y)+\mathcal{O}(2k+1).
\end{array}\]
This completes the proof.
\end{proof}

\begin{remark}
Let $\gamma, \gamma'$ be oriented ordered $4$-component bottom tangles illustrated in Figure~\ref{2k+2}.
Note that $\gamma'$ is obtained from $\gamma$ by a SL-move. 
By the definition, $\mu_{\gamma}(j)=\mu_{\gamma'}(j)=0$ for any sequence $j$ with $|j|=1(=k)$.
By Lemma \ref{lemma2}, $\mu_{\gamma}(J)=\mu_{\gamma'}(J)$ for any sequence $J$ with $|J|\leq 3(=2k+1)$.
However, by easy calculus, $\mu_{\gamma}(1234)=0\neq 1=\mu_{\gamma'}(1234)$ for the sequence 1234 of the length 4(=2k+2).
Therefore Lemma \ref{lemma2} (Theorem \ref{well-definedness}) generally dose not hold in the case of the length $2k+2$ or more.
\end{remark}

\begin{figure}[htbp]
  \begin{center}
    \begin{overpic}[width=10cm]{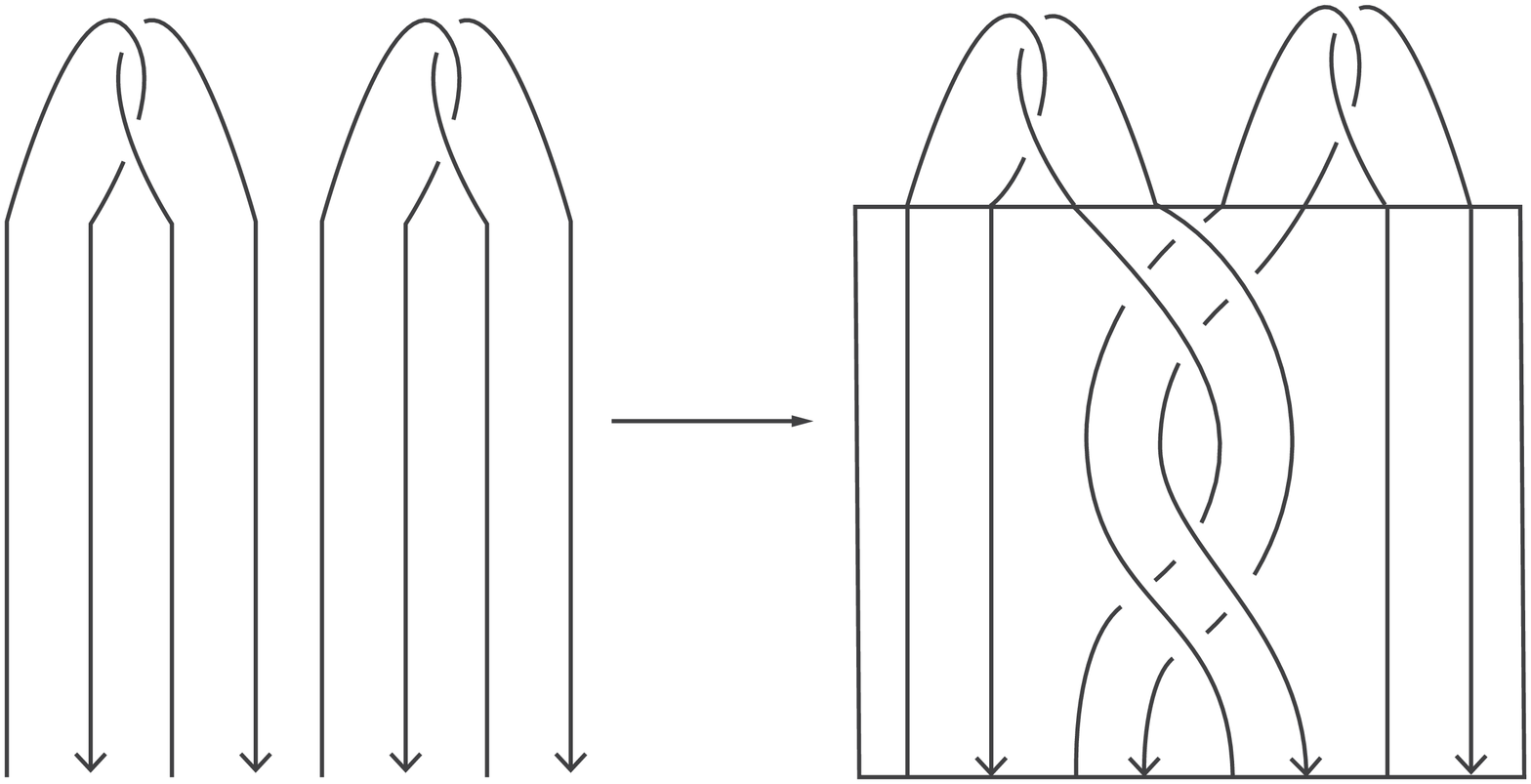}
    \linethickness{3pt}
    \put(50,-8){$\gamma$}
    \put(217,-8){$\gamma'$}
    \put(113,70){SL-move}
    \put(6,140){1}
    \put(37,140){2}
    \put(65,140){3}
    \put(95,140){4}
    \put(175,140){1}
    \put(206,140){2}
    \put(232,140){3}
    \put(264,140){4}
    \end{overpic}
    \caption{An example of Lemma \ref{lemma2} (Theorem \ref{well-definedness}) dose not hold in the case of the length $2k+2$ or more}
    \label{2k+2}
  \end{center}
\end{figure}

By using the above results, we prove Theorem \ref{well-definedness}.
\begin{proof}[Proof of Theorem \ref{well-definedness}]
By Proposition \ref{propPB-move}, any two disk/band surfaces $F_{c}$ and $F'_{c}$ of an $n$-component clover link $c$ are transformed into each other by the moves (a) and (b) illustrated in Figure \ref{PB-move}.
So two bottom tangles $\gamma_{F_{c}}$ and $\gamma_{F'_{c}}$ are transformed into each other by a SL-move.
Since the both closures $L(\gamma_{F_{c}})$ and $L(\gamma_{F'_{c}})$ are ambient isotopic to $l_{c}$ and the hypothesis of Lemma~\ref{lemma2}, 
\[0=\omu_{l_{c}}(J)=\mu_{\gamma_{F_{c}}}(J)=\mu_{\gamma_{F'_{c}}}(J)\]
 for any sequence $J$ with $|J|\leq k.$
Hence by Lemma \ref{lemma2}, $\mu_{\gamma_{F_{c}}}(I)=\mu_{\gamma_{F'_{c}}}(I)$ for any sequence $I$ with $|I|\leq 2k+1$.
This completes the proof.
\end{proof}
\section{Edge-homotopy for clover links}
\subsection{Claspers}
Let us briefly recall from \cite{H} the basic notations of clasper theory.
In this paper, we essentially only need the notation of $C_k$-tree.
For a general definition of claspers, we refer the reader to \cite{H}.
\begin{definition}
Let $L$ be a (clover) link in $S^{3}$ (resp. a tangle in $[0,1]^{3}$). 
An embedded disk $T$ in $S^{3}$ (resp. $[0,1]^{3}$) is called a {\it tree clasper} for $L$ if it satisfies the following (1), (2) and (3):\\
(1) $T$ is decomposed into disks and bands, called {\it edges}, each of which connects two distinct disks.\\
(2) The disks have either 1 or 3 incident edges, called {\it disk-leaves} or {\it nodes} respectively.\\
(3) $L$ intersects $T$ transversely and the intersections are contained in the union of the interior of the disk-leaves.\\
The {\it degree} of a tree clasper is the number of the disk-leaves minus 1.
(In \cite{H}, a tree clasper is called a {\it strict tree clasper}.)
A degree $k$ tree clasper is called a {\it $C_k$-tree}.
A $C_k$-tree is {\it simple} if each disk-leaf intersects $L$ at one point.
\end{definition}

We will make use of the drawing convention for claspers of \cite[Fig. 4]{H}.
Given a $C_k$-tree $T$ for $L$, there is a procedure to construct a framed link $\gamma(T)$ in a regular neighborhood of $T$.
{\it Surgery along $T$} means surgery along $\gamma(T)$.
Since surgery along $\gamma(T)$ preserves the ambient space, surgery along the $C_k$-tree $T$ can be regarded as a local move on $L$ in $S^{3}$ (resp. $[0,1]^{3}$).
We say that the {\it resulting one $L_T$ in $S^{3}$ {\rm (}resp. $[0,1]^{3}${\rm )} is obtained from $L$ 
by surgery along $T$}.
Similarly, for a disjoint union of trees $T_1 \cup \cdots \cup T_m$, we can define $L_{T_1\cup \cdots \cup T_m}$.
Two (string) links (resp. a clover link) are {\it $C_{k}$-equivalence} if they are transformed into each other by surgery along $C_{k}$-trees.

The relation between $C_{k}$-equivalence and Milnor invariants is known as follows.
\begin{theorem}\cite[Theorem 7.2]{H}
\label{thmH}
Milnor invariants of length $\leq k$ for {\rm (}string{\rm )} links are invariants of $C_k$-equivalence. 
\end{theorem}

A $C_k$-tree $T$ for a (string) link (resp. a clover link) $L$ is a {\it self $C_{k}$-tree} if all disk-leaves of $T$ intersect the same component (resp. the same spatial edge) of $L$.
The {\it self $C_{k}$-equivalence} is an equivalence relation generated by surgery along self $C_{k}$-trees.
We remark that self $C_{1}$-equivalence for a (string) link (resp. a clover link) is the same relation as link-homotopy (resp. edge-homotopy).

\subsection{Edge-homotopy for clover links}
The (edge-homotopy+$C_{k}$)-equivalence is an equivalence relation obtained by combining edge-homotopy and $C_{k}$-equivalence.
\begin{theorem}
\label{mainthm}
Let $c, c'$ be two $n$-clover links and $l_{c}, l_{c'}$ links which are disjoint unions of 
leaves of $c, c'$ respectively.
Suppose that  $\omu_{l_{c}}(J)= \omu_{l_{c'}}(J)=0$ for any sequence $J$ with $|J| \leq k$. 
Then $c$ and $c'$ are {\rm (}edge-homotopy$+C_{2k+1}${\rm )}-equivalence if and only if $\mu_c(I)= \mu_{c'}(I)$ for any non-repeated sequence $I$ with $|I| \leq 2k+1$.
\end{theorem}

Let $c$ be a clover link and $l_{c}$ a link which is the disjoint union of leaves of $c$.
Since the union of stems of $c$ and the root is contractible,
we may assume that a $C_{k}$-tree $T$ for $c$ satisfies $T\cap c=T\cap l_{c}$.
By the zip construction~\cite{H} for $T$, $T$ becomes a disjoint union of simple $C_{k}$-trees for $l_{c}$.
Combining this and \cite[Lemma 1.2]{FY}, 
for $n\leq m$, if two $n$-clover links are $C_{m}$-equivalence, then they are 
self $C_{1}$-equivalence. 
Hence by Theorem~\ref{mainthm}, we have the following corollary.

\begin{corollary}
\label{cor}
Let $c, c'$ be two $n$-clover links and $l_{c}, l_{c'}$ links which are disjoint unions of 
leaves of $c, c'$ respectively.
Suppose that $\omu_{l_{c}}(J)= \omu_{l_{c'}}(J)=0$ for any sequence $J$ with $|J| \leq n/2$. 
Then $c$ and $c'$ are edge-homotopic if and only if $\mu_c(I)= \mu_{c'}(I)$ for any non-repeated sequence $I$ with $|I| \leq n$.
\end{corollary}

By the definition, the Milnor $\omu$-invariant of length $1$ is always zero.
If $n=3$, then Corollary \ref{cor} is not necessary the condition.
\begin{corollary}
\label{corn=3}
Two 3-clover links $c$ and $c'$ are edge-homotopic if and only if $\mu_{c}(I)~=~\mu_{c'}(I)$ for any non-repeated sequence $I$ with $|I|\leq 3$.
\end{corollary}

\begin{example}
\label{example}
Let $c$ and $c'$ are two $n$-clover links as described in Figure~\ref{non-edge-homotopic}.
Let $l_{c}, l_{c'}$ be links that are disjoint unions of leaves of $c, c'$ respectively.
It is clear that $l_{c}$ and $l_{c'}$ are ambient isotopic.
However, $c$ and $c'$ are not edge-homotopic by Corollary \ref{cor} because $\mu_{c}(123)=1\neq0=\mu_{c'}(123)$.
\end{example}

\begin{figure}[htbp]
  \begin{center}
    \begin{overpic}[width=10cm]{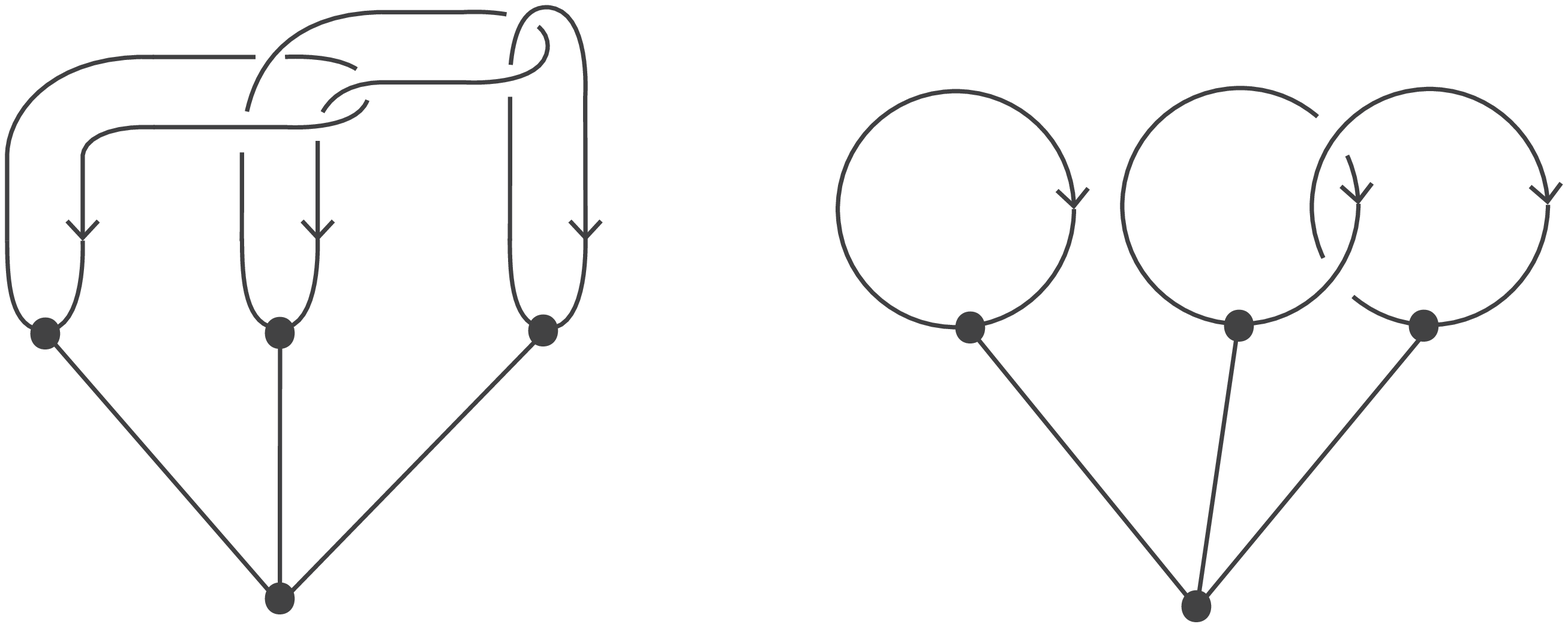}
    \linethickness{3pt}
    \put(-14,85){$l_{c}$}
    \put(138,85){$l_{c'}$}
    \put(5,131){1}
    \put(48,131){2}
    \put(96,131){3}
    \put(173,117){1}
    \put(223,117){2}
    \put(260,117){3}
    \put(49,5){$c$}
    \put(216,5){$c'$}
    \end{overpic}
    \caption{An example of two clover links which are not edge-homotopic to each other}
    \label{non-edge-homotopic}
  \end{center}
\end{figure}

\section{Proof of Theorem \ref{mainthm}}
In order to prove Theorem~\ref{mainthm}, we need 
the following lemma and a theorem given by \cite{Y}.
\begin{lemma}
\label{lemma3}
Let $c, c'$ be two clover links.
Then $c$ and $c'$ are {\rm (}edge-homotopy$+C_k${\rm )}-equivalence if and only if there exist disk/band surfaces $F_c, F_{c'}$ of $c, c'$ respectively such that the bottom tangles $\gamma_{F_c}$ and $\gamma_{F_{c'}}$ are {\rm (}link-homotopy$+C_k${\rm )}-equivalence.
\end{lemma}

\begin{proof}
We first prove the sufficient condition.
By the assumption, there exists a disjoint union $G$ of self $C_1$-trees and $C_k$-trees for $c$ such that $c'=c_G$. 
For $c, c'$, let $F_c, F_{c'}$ be disk/band surfaces of $c, c'$ respectively. 
We may assume that $F_c\cap G$ is contained in the interior of $G$.
Then we have $\partial F_{c'}=(\partial F_c)_G$ if necessary adding full-twists to bands of $F_c$. 
Therefore $\gamma_{F_{c'}}=(\gamma_{F_c})_G$ and $G$ consists of self $C_{1}$-trees and $C_{k}$-trees for $\gamma_{F_{c}}$.

We next prove the necessary condition.
Let $S$ be the union of stems and the root of $c$, and let $N(S)$ be the regular neighborhood of $S$.
Then we may assume that
$S$ is equal to the union of stems and the root of $c'$
and $c\setminus N(S)=\gamma_{F_c}$.
Similarly $c'\setminus N(S)=\gamma_{F_{c'}}$.
By the assumption, $\gamma_{F_c}$ and $\gamma_{F_{c'}}$ are
 (link-homotopy$+C_k$)-equivalence. This completes the proof.
\end{proof}

Let $\pi :\{ 1,\ldots ,k\} \rightarrow \{ 1,\ldots ,n\} (k \leq n)$ be an injection such that $\pi(1)<\pi(i)<\pi(k) (i\in \{2,\ldots ,k-1\})$, and ${\mathcal F}_k$ be the set of such injections.
For $\pi \in {\mathcal F}_k$, let $T_{\pi}$ and $\overline{T}_{\pi}$ be simple $C_{k-1}$-trees as illustrated in Figure \ref{Ck-trees}, and set $V_{\pi}$ (resp. $V_{\pi}^{-1}$) be a string link obtained from the $n$-component trivial string link by surgery along $T_{\pi}$ (resp. $\overline{T}_{\pi}$).
Here, Figure~\ref{Ck-trees} are the images of homeomorphisms from the neighborhood of $T_{\pi}$ and $\overline{T}_{\pi}$ to $B^{3}$.
\begin{figure}[htbp]
  \begin{center}
     \begin{overpic}[width=10cm]{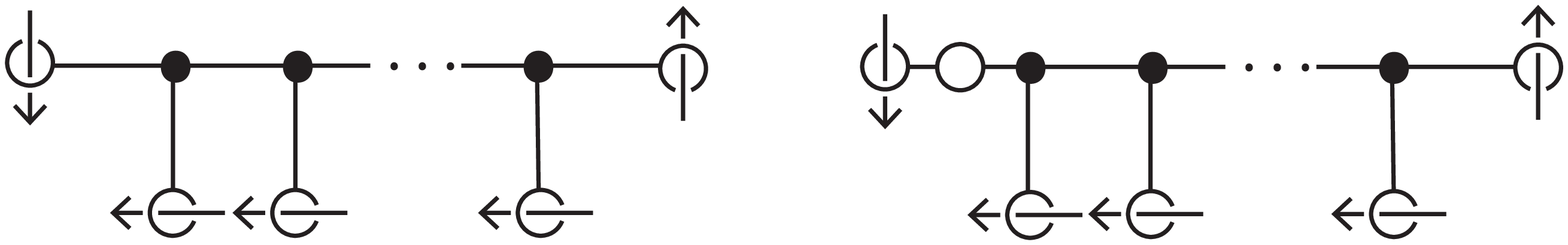}
    \linethickness{3pt}
    \put(60,50){$T_{\pi}$}
    \put(215,50){$\overline{T}_{\pi}$}
    \put(171.7,29.5){$s$}
    \put(-4,50){$\pi(1)$}
    \put(22,-8){$\pi(2)$}
    \put(45,-8){$\pi(3)$}
    \put(78,-8){$\pi(k-1)$}
    \put(115,50){$\pi(k)$}
    \put(151,50){$\pi(1)$}
    \put(176,-8){$\pi(2)$}
    \put(200,-8){$\pi(3)$}
    \put(232,-8){$\pi(k-1)$}
    \put(271,50){$\pi(k)$}
    \end{overpic}
    \caption{}
    \label{Ck-trees}
  \end{center}
\end{figure}

\begin{theorem}\cite[Theorem 4.3]{Y}.
\label{thmY}
Let $sl$ be an $n$-component string link.
Then $sl$ is link-homotopic to $sl_1*sl_2*\cdots*sl_{n-1}$, where
\begin{eqnarray*}
sl_i&=&\prod \limits_{\pi \in {\mathcal F}_{i+1}} V_{\pi}^{x_{\pi}}, \\
x_{\pi}&=&
\begin{cases}
\mu_{sl}(\pi(1)\pi(2)) &\text{$(i=1)$},\\
\mu_{sl}(\pi(1)\ldots \pi(i+1))-\mu_{sl_1*\cdots *sl_{i-1}}(\pi(1)\ldots \pi(i+1)) &\text{$(i\geq 2)$}.
\end{cases}
\end{eqnarray*}
\end{theorem}

By using the results above, we prove Theorem \ref{mainthm}.
\begin{proof}[Proof of Theorem \ref{mainthm}]
We first prove the sufficient condition.
Since $c$ and $c'$ are (edge-homotopy$+C_{2k+1}$)-equivalence, by Lemma \ref{lemma3}, there exist disk/band surfaces $F_c, F_{c'}$ of $c, c'$ respectively such that the $n$-component bottom tangles $\gamma_{F_c}$ and $\gamma_{F_{c'}}$ are (link-homotopy+$C_{2k+1}$)-equivalence.
Since the Milnor number with a non-repeated sequence is an invariant of link-homotopy \cite{HL}, by Theorem \ref{thmH}, $\mu_c(I)= \mu_{c'}(I)$ with a non-repeated sequence $I$ with $|I| \leq 2k+1$.

Let us prove the necessary condition.
Let $F_{c}, F_{c'}$ be disk/band surfaces of $c, c'$ respectively and let be the $\gamma_{F_{c}}, \gamma_{F_{c'}}$ be $n$-component bottom tangles.
By Theorem \ref{thmY}, two string links which correspond to $\gamma_{F_{c}}$ and $\gamma_{F_{c'}}$ are link-homotopic to $sl_{1}*sl_{2}*\cdots *sl_{n-1}$ and $sl'_{1}*sl'_{2}*\cdots *sl'_{n-1}$ respectively.
Both $sl_{i}$ and $sl'_{i}$ are $C_{i}$-equivalent to an $n$-component trivial string link.
Set $m=\min\{ n, 2k+1\} $, by Theorem~\ref{thmY}, 
the two string links which correspond to $\gamma_{F_{c}}$ and $\gamma_{F_{c'}}$ are (link-homotopy+$C_{2k+1}$) to $sl_{1}*sl_{2}*\cdots *sl_{m-1}$ and $sl'_{1}*sl'_{2}*\cdots *sl'_{m-1}$ respectively.
We recall that if $n\leq 2k+1$, $C_{2k+1}$-equivalence implies link-homotopy. 
By the assumption, since $\mu_{\gamma_{F_{c}}}(I)=\mu_{\gamma_{F_{c'}}}(I)$ for any non-repeated sequence $I$ with $|I|\leq 2k+1$, we have $sl_{i}=sl'_{i}$ for each $i~(\leq m-1)$.
This implies $\gamma_{F_c}$ and $\gamma_{F_{c'}}$ are (link-homotopy$+C_{2k+1}$)-equivalence.
Lemma \ref{lemma3} completes the proof.
\end{proof}


\end{document}